\documentclass[10pt,reqno]{amsart}
   \usepackage[centertags]{amsmath}
   \usepackage{amsfonts}
   \usepackage{amsmath}
   \usepackage{amssymb}
   \usepackage{amsthm}
   \usepackage{newlfont}
\usepackage[latin1]{inputenc}%%%%para que reconozca las tildes.

\usepackage{multirow, bigdelim}

\usepackage{epsfig}
\usepackage{pst-all}
\usepackage{pstricks}

\usepackage{color}

\usepackage{graphicx}

\usepackage{bm}

\allowdisplaybreaks[3]
\theoremstyle{plain}
\newtheorem{theorem}{Theorem}[section]
\newtheorem{lemma}[theorem]{Lemma}

\theoremstyle{definition}

\theoremstyle{remark}
\newtheorem{remark}[theorem]{Remark}
\newtheorem*{remark*}{Remark}

\numberwithin{equation}{section}

\title[QBD and multivariate OP]{Quasi-birth-and-death processes and \\ multivariate orthogonal polynomials}

\author{Lidia Fern\'andez and Manuel D. de la Iglesia}

\address{Lidia Fern\'andez\\
IEMATH-GR and Departamento de Matem\'atica Aplicada\\
Universidad de Granada\\
18071, Granada, Spain.}
\email{lidiafr@ugr.es}
\address{Manuel D. de la Iglesia\\
Instituto de Matem\'aticas \\
Universidad Nacional Aut\'onoma de M\'exico \\
Circuito Exterior, C.U.\\
04510, Mexico D.F. Mexico.}
\email{mdi29@im.unam.mx}
\thanks{The work of the first author was partially supported by MICINN of Spain and European Regional Development Fund (ERDF) through the grant PGC2018-094932-B-I00, and Research Group FQM-384 by Junta de Andaluc\'{i}a. 
The work of the second author was partially supported by PAPIIT-DGAPA-UNAM grant IN104219 (M\'exico) and CONACYT grant A1-S-16202 (M\'exico).}

\date{\today}



\begin{document}

\maketitle

\begin{abstract}
The aim of this paper is to study some models of quasi-birth-and-death (QBD) processes arising from the theory of bivariate orthogonal polynomials. First we will see how to perform the spectral analysis in the general setting as well as to obtain results about recurrence and the invariant measure of these processes in terms of the spectral measure supported on some domain $\Omega\subset\mathbb{R}^d$. Afterwards, we will apply our results to several examples of bivariate orthogonal polynomials, namely product orthogonal polynomials, orthogonal polynomials on a parabolic domain and orthogonal polynomials on the triangle. We will focus on linear combinations of the Jacobi matrices generated by these polynomials and produce families of either continuous or discrete-time QBD processes. Finally, we show some urn models associated with these QBD processes.
\end{abstract}

\section{Introduction}

The connection between one-dimensional birth-death models and orthogonal polynomials goes back to the pioneering work of S. Karlin and J. McGregor \cite{KMc2, KMc3, KMc6}. In a series of papers they established an important connection between the transition probability functions of continuous-time birth-death processes and discrete-time birth-death chains (in this order) by means of a spectral representation, the so-called \emph{Karlin-McGregor integral representation formula}. This representation is possible  since the one-step transition probability matrix of the birth-death chain or the infinitesimal operator of the birth-death process are tridiagonal matrices, so we can apply the spectral theorem to find the corresponding spectral measure associated with the process. Many probabilistic aspects can be analyzed in terms of the corresponding orthogonal polynomials, such as transition probabilities, the invariant measure or the recurrence of the process. In the last 60 years, many other authors e.g. M. Ismail, G. Valent, P. Flajolet, F. Guillemin, H. Dette or E. van Doorn, to mention a few, have studied this connection and other probabilistic aspects. For a brief account of all these relations see \cite{Scho}.

A natural extension in this direction are the so-called \emph{quasi-birth-and-death} (QBD) \emph{processes}. The state space, instead of $\mathbb{N}_0$, is given by pairs of the form $(n,k)$, where $n\in\mathbb{N}_0$ is usually called the \emph{level}, while $1\leq k\leq r_n$ is referred to as the \emph{phase}. Observe that the number of phases may depend on the different levels. For a general setup see \cite{BTa, LaR}. Now the QBD process, at each time step, is restricted to move only between adjacent levels but transitions between phases are all possible. That means that the transition probability matrix (discrete-time) or the infinitesimal operator matrix (continuous-time) of the QBD process is then block-tridiagonal of the form \eqref{PP} (see below), also known as a \emph{block Jacobi matrix}. If $r_n=1$ for all $n\in\mathbb{N}_0$ then we go back to the classical birth-death chain. If $r_n=N$ for all $n\in\mathbb{N}_0$, where $N$ is a positive integer, then all blocks in the Jacobi matrix have the same dimension $N\times N$. In this case, the spectral analysis can be performed using \emph{matrix-valued orthogonal polynomials} (see \cite{DRSZ, G2} for the discrete-time case and \cite{DR} for the continuous-time case). In the last years, many new examples related to matrix-valued orthogonal polynomials have been analyzed by using spectral methods (see \cite{Clay, G2, G1, G5, GdI3, dI1, dIR}).

As it was mentioned in Section 5 of \cite{Clay}, a natural source of examples of more complicated QBD processes may come from the theory of multivariate orthogonal polynomials. These polynomials can be defined in terms of a positive linear functional $\mathcal{L}$ which we assume it is expressible as integrals with respect to a nonnegative weight function $w$ with finite moments supported on some domain $\Omega\subset\mathbb{R}^d$. If we start with a weight function $w$ then the corresponding multivariate orthogonal polynomials satisfy $d$ different three-term recurrence relations (see \eqref{TTRRvv} below). For each $1\leq i\leq d$, the coefficients of these recurrence relations can be written in block tridiagonal form (or block Jacobi matrix) $J_i$ and have the same structure as in \eqref{PP} (see below). The goal of this paper is to find appropriate normalizations of the multivariate orthogonal polynomials such that linear combinations of the corresponding Jacobi matrices $J_i$ of the form $\tau_1J_1+\cdots+\tau_dJ_d$ give rise to discrete or continuous-time QBD processes. In particular, we will study several examples of bivariate orthogonal polynomials.

Multivariate orthogonal polynomials have appeared before in the literature in connection with probabilistic and stochastic models. The first examples probably appeared in the study of some stochastic models in genetics \cite{KMc8}, Ehrenfest urn models \cite{KMc11} or linear growth models \cite{KMc7, Mil}. After that, many other authors like P. Diaconis, R. Griffiths, F.A. Gr\"unbaum or M. Rahman have found other connections between multivariate orthogonal polynomials and probabilistic models like the multinomial distribution \cite{DiGr, Grif}, Lancaster distributions \cite{Grif2}, composition birth-death processes \cite{Grif3} or poker dice games \cite{GR}. The multivariate orthogonal polynomials involved in these applications are always \emph{discrete}. Our approach is different since we will start from very well known examples of bivariate \emph{continuous} orthogonal polynomials and try to generate families of QBD processes from certain linear combinations of the Jacobi matrices generated by these polynomials.

The paper is divided into two parts. First part (Section \ref{sec1}) comprises an extension of the results from Section 5 of \cite{Clay} for a particular class of bivariate orthogonal polynomials to the general setting. Besides, we obtain other important results related with the invariant measure and the recurrence of the QBD processes. In the second part, we will apply our results to several examples of QBD processes generated by bivariate orthogonal polynomials. In Section \ref{SecPOP} we consider product orthogonal polynomials such as the product Jacobi and Laguerre polynomials and we show that the QBD processes have independent components. In Section \ref{secpar} we will study a family of QBD processes associated with orthogonal polynomials on a parabolic domain. The two components of the QBD process are now dependent. In particular, we will give an urn model associated with one particular situation. In Section \ref{sectri} we will study a family of QBD processes associated with orthogonal polynomials on the triangle. The transitions between the bivariate states are much more involved in this situation. Nevertheless, we will be able to give an urn model by considering a stochastic block LU factorization of the Jacobi matrix, in the same spirit as the one used in \cite{GdI1, GdI2}. Finally, we finish in Section \ref{secult} with some concluding remarks and suggestions for further research.

\section{QBD processes and multivariate orthogonal polynomials}\label{sec1}

Let $\{Z_t : t\geq0\}$ be a time-homogeneous Markov chain on the state space of pairs $(n,k)$ where $n\in\mathbb{N}_0$ is usually called the \emph{level} and $1\leq k\leq r_n$ is usually called the ($n$-dependent) \emph{phase}. We say that $Z_t$ is a \emph{quasi-birth-and-death (QBD) process} if the only allowed transitions are between \emph{adjacent levels}, but transitions between phases are all possible.

If we have a discrete-time QBD process $\{Z_t : t=0,1,\ldots\}$ this condition is equivalent to
$$
\mathbb{P}\left[Z_1=(n_1,k_1) \; |\; Z_0=(n_0,k_0)\right]=0,\quad |n_1-n_0|>1,
$$
The one step transition probability matrix $\bm P$ has a block-tridiagonal form
\begin{equation}\label{PP}
\bm P=\left(\begin{array}{cccccc}
\bm B_{0}   & \bm A_{0}   &         &       &\bigcirc \\
\bm C_{1}   & \bm B_{1}   & \bm A_{1} &       &          \\
          & \bm C_{2}   & \bm B_{2} & \bm A_2 &         \\
  \bigcirc         &           & \ddots  & \ddots & \ddots
\end{array}\right),
\end{equation}
where $\bm A_i$, $\bm B_i$ and $\bm C_i$ are matrices of dimension $r_i\times r_{i+1}$, $r_i\times r_{i}$ and $r_i\times r_{i-1}$, respectively. The symbol $\bigcirc$ stands for block zero matrices which fill the remaining entries. In the entries of the matrix $\bm A_i$ we can find the probabilities of all the different ways of moving up one level while going from any phase to any other phase, starting at level $i$. The number of phases $r_i$ depend on the level $i$. The same interpretation applies for the coefficients $\bm B_i$ (staying at the same level) and $\bm C_i$ (moving down one level). If $r_i=1$ for all $i$ then we recover the classical discrete-time birth-death chain on $\mathbb{N}_0$.

Let us denote by $\bm e_N$ the $N$-dimensional vector with all components equal to 1, i.e.
\begin{equation}\label{eeN}
\bm e_N=(1,1,\ldots,1)^T,
\end{equation}
and we will also use the notation $\bm e=\bm e_\infty$. Since $\bm P$ is a stochastic matrix we have nonnegative (scalar) entries
and all (scalar) rows add up to one, i.e. $\bm P\bm e=\bm e$. In other words,
\begin{equation}\label{stochcond}
\bm B_0\bm e_{r_0}+\bm A_0\bm e_{r_1}=\bm e_{r_0},\quad \bm C_i\bm e_{r_{i-1}}+\bm B_i\bm e_{r_i}+\bm A_i\bm e_{r_{i+1}}=\bm e_{r_i},\quad i\geq1.
\end{equation}
If we have a continuous-time QBD process then we will assume that there exists a conservative infinitesimal operator $\bm{\mathcal{A}}$ associated with the transition probability function $\bm P(t)$ and it has the same block tridiagonal structure as in \eqref{PP}. That means that all off-diagonal (scalar) entries are nonnegative and all (scalar) rows add up to 0, i.e $\bm{\mathcal{A}}\bm e=\bm 0$. In other words,
\begin{equation*}\label{stochcond2}
\bm B_0\bm e_{r_0}+\bm A_0\bm e_{r_1}=\bm 0,\quad \bm C_i\bm e_{r_{i-1}}+\bm B_i\bm e_{r_i}+\bm A_i\bm e_{r_{i+1}}=\bm 0,\quad i\geq1.
\end{equation*}
The transition probability function $\bm P(t)$ with $\bm P(0)=\bm I$ and $\bm P'(0)=\bm{\mathcal{A}}$ satisfies the Kolmogorov equations
$$
\bm P'(t)=\bm{\mathcal{A}}\bm P(t)=\bm P(t)\bm{\mathcal{A}},\quad t\geq0.
$$

Our goal is to relate transition probabilities matrices or infinitesimal operators of the form \eqref{PP} with the theory of multivariate orthogonal polynomials and viceversa. In particular, we will focus on examples already known in the theory of multivariate orthogonal polynomials from which we can derive block tridiagonal matrices of the form \eqref{PP} with probabilistic properties. If we start with \eqref{PP} and we want to use the spectral theorem for multivariate orthogonal polynomials, then we will have to assume several hypothesis, as we will see now. This approach, already introduced in \cite{Clay} for the case of bivariate orthogonal polynomials and coefficients $\bm B_n=0$, $\bm A_n$ upper bidiagonal and $\bm C_n$ lower bidiagonal, relies on the spectral theory of commuting self-adjoint operators (see \cite{DX14}) and it is different from previous approaches (see \cite{KMc7, Mil}). Although we will be interested in finding QBD processes from very well known examples of bivariate orthogonal polynomials, we will show the general case of multivariate orthogonal polynomials generated by \eqref{PP}.

Let us denote by $\mathbb{R}[x_1,\ldots,x_d]$ the ring of polynomials with $d\in\mathbb{N}$. For each level $n\in\mathbb{N}_0,$ the number of phases will depend on $n$ and $d$ in the following form:
\begin{equation}\label{knd}
r_n=r_n^d=\binom{n+d-1}{n}.
\end{equation}
This number is just the dimension of the space of all homogeneous polynomials of total degree $n$ in $\mathbb{R}[x_1,\ldots,x_d]$. Let us assume that we can write $\bm P$ (or $\bm{\mathcal{A}}$) in \eqref{PP} in the following way
\begin{equation}\label{Palpha}
\bm P=\tau_1 J_1+\cdots+\tau_d J_d,
\end{equation}
where $\tau_i\in\mathbb{R}, i=1,\ldots,d$ (to be determined depending on the example) and $J_i$ are block tridiagonal matrices of the form \eqref{PP} with coefficients $C_{n+1,i}, B_{n,i}, A_{n,i}, i=1,\ldots,d, n\ge 0,$  of the same dimension as in $\bm P$. We will denote by $C_{n+1}^T, B_n, A_n, n\geq0,$ the \emph{joint matrices} associated with $C_{n+1,i}^T, B_{n,i}, A_{n,i}, i=1,\ldots,d$, i.e. the row block column vectors built from these coefficients. All these matrices are subject to the following rank conditions:
\begin{align}
\label{rankc1} \mbox{rank} (A_{n,i})&=\mbox{rank} (C_{n+1,i})=r_n^d,\\
\label{rankc2} \mbox{rank} (A_{n})&=\mbox{rank} (C_{n+1}^T)=r_{n+1}^d.
\end{align}
Since $A_n$ has full rank, it has a generalized inverse, which we denote by $D_n^T=(D_{n,1}^T\, \cdots \,D_{n,d}^T)$. Therefore, we have
\begin{equation}\label{geninvA}
D_n^TA_n=\sum_{i=1}^dD_{n,i}^TA_{n,i}=I.
\end{equation}
From here we can construct recursively a family of multivariate polynomials $(\mathbb{P}_n)_{n\geq0}$ where $\mathbb{P}_n=(P_{n,1},\ldots,P_{n,r_n^d})^T$ using Theorem 3.3.5 of \cite{DX14}, by the following formula:
\begin{equation*}\label{pols}
\mathbb{P}_{n+1}(x)=\sum_{i=1} ^dx_iD_{n,i}^T\mathbb{P}_n(x)+E_n\mathbb{P}_n(x)+F_n\mathbb{P}_{n-1}(x),
\end{equation*}
where
$$
E_n=-\sum_{i=1}^dD_{n,i}^TB_{n,i},\quad F_n=-\sum_{i=1}^dD_{n,i}^TC_{n,i}^T.
$$
Here $x=(x_1,\ldots,x_d)$. Finally, in order to apply the spectral theorem for commuting self-adjoint operators, we need to assume first the following commutativity conditions
\begin{equation*}\label{commc}
J_iJ_j=J_jJ_i,\quad\mbox{for all}\quad i,j=1,\ldots,d,
\end{equation*}
and second that we can ``symmetrize'' in some way each one of the operators $J_i$. For that, we will have to assume that there exists a sequence of nonsingular matrices $(S_n)_{n\geq0}$, each of dimension $r_n^d\times r_n^d$ such that
\begin{equation}\label{symm1}
S_nB_{n,i}S_n^{-1}\quad\mbox{is symmetric}\quad n\geq0,\quad i=1,\ldots,d,
\end{equation}
and
\begin{equation}\label{symm2}
A_{n,i}S_{n+1}S_{n+1}^T=S_nS_n^TC_{n+1,i}^T, \quad n\geq0,\quad i=1,\ldots,d.
\end{equation}

Under all these hypotheses we can guarantee (see Theorem 3.5.1 of \cite{DX14}) that there exists a positive definite linear functional $\mathcal{L}$ such that
\begin{equation*}
\begin{split}
\mathcal{L}(\mathbb{P}_i\mathbb{P}_j^T)&=\bm 0_{r_i\times r_j},\quad i\neq j,\\
\mathcal{L}(\mathbb{P}_j\mathbb{P}_j^T)&=\Pi_j^{-1},\quad \Pi_j^{-1}=S_jS_j^T.
\end{split}
\end{equation*}
If we assume that $\mathcal{L}$ is expressible as integrals with respect to a (scalar-valued) nonnegative weight function $w(x)$ with finite moments supported on $\Omega\subset\mathbb{R}^d$, then we have
\begin{equation}\label{norms}
\begin{split}
\int_\Omega\mathbb{P}_i(x)\mathbb{P}_j^T(x)w(x)dx&=\bm 0_{r_i\times r_j},\quad i\neq j,\\
\int_\Omega\mathbb{P}_j(x)\mathbb{P}_j^T(x)w(x)dx&=\Pi_j^{-1},\quad \Pi_j^{-1}=S_jS_j^T.
\end{split}
\end{equation}
In particular, the multivariate polynomials $(\mathbb{P}_n)_{n\geq0}$ satisfies the three-term recurrence relations
\begin{equation}\label{TTRRvv}
x_i\mathbb{P}_n(x)=A_{n,i}\mathbb{P}_{n+1}(x)+B_{n,i}\mathbb{P}_n(x)+C_{n,i}\mathbb{P}_{n-1}(x),\quad n\geq0,\quad i=1,\ldots,d,\quad \mathbb{P}_{-1}=0.
\end{equation}

The strong tool of the spectral theorem for commuting self-adjoint operators allows us to derive the analogue of the \emph{Karlin-McGregor integral representation formula}. If we have a discrete-time QBD process, then this formula gives an expression of the $(i,j)$ block of the matrix $\bm P^n$ in terms of the multivariate orthogonal polynomials. Indeed,
\begin{equation}\label{KMcF}
\bm P_{i,j}^n=\left(\int_\Omega(\tau_1x_1+\cdots+\tau_dx_d)^n\mathbb{P}_i(x)\mathbb{P}_j^T(x)w(x)dx\right)\Pi_j.
\end{equation}
Observe that each block $\bm P_{i,j}^n$ is of dimension $r_i^d\times r_j^d$ and the entries of this block gives all probabilities of moving from one phase to any other. In the case when the family of polynomials $(\mathbb{P}_n)_{n\geq0}$ is mutually orthogonal (and therefore $\Pi_j$ is a diagonal matrix with diagonal entries $\Pi_{j,k}, k=1,\ldots,r_j^d$) we have a compact way of expressing these probabilities by the following formula
\begin{equation}\label{KMcFd}
\mathbb{P}\left[Z_n=(j,j') \; |\; Z_0=(i,i')\right]=\left(\bm P_{i,j}^n\right)_{i',j'}=\Pi_{j,j'}\int_\Omega\left(\sum_{k=1}^d\tau_kx_k\right)^nP_{i,i'}(x)P_{j,j'}(x)w(x)dx.
\end{equation}

If we have a continuous-time QBD process then this formula gives an expression of the $(i,j)$ block of the transition function $\bm P(t)$ in terms of the multivariate orthogonal polynomials. Indeed,
\begin{equation}\label{KMcF2}
\bm P_{i,j}(t)=\left(\int_\Omega\mbox{exp}\left((\tau_1x_1+\cdots+\tau_dx_d)t\right)\mathbb{P}_i(x)\mathbb{P}_j^T(x)w(x)dx\right)\Pi_j.
\end{equation}
Again, each block $\bm P_{i,j}(t)$ is of dimension $r_i^d\times r_j^d$ and if the family of polynomials $(\mathbb{P}_n)_{n\geq0}$ is mutually orthogonal, then we have
\begin{equation}\label{KMcFd2}
\mathbb{P}\left[Z_t=(j,j') \; |\; Z_0=(i,i')\right]=\left(\bm P_{i,j}(t)\right)_{i',j'}=\Pi_{j,j'}\int_\Omega\mbox{exp}\left(\sum_{k=1}^d\tau_kx_k\right)P_{i,i'}(x)P_{j,j'}(x)w(x)dx.
\end{equation}

The case of regular discrete-time birth-death chain can be found in \cite{KMc6}, while the case of regular birth-death processes can be found in \cite{KMc2,KMc3}.

\medskip

On the contrary, if we have a nonnegative weight function $w(x)$ with finite moments supported on some domain $\Omega\subset\mathbb{R}^d$ then it is possible to construct a family of multivariate polynomials $(\mathbb{P}_n)_{n\geq0}$ satisfying \eqref{TTRRvv}, where the coefficients can be computed in terms of the linear functional generated by the weight function (see Theorem 3.3.1 of \cite{DX14}). All examples we will see in this paper are of this form.

The sequence of ``norms'' $(\Pi_n^{-1})_{n\geq0}$ in \eqref{norms}, where each $\Pi_n$ is a nonsingular matrix of dimension $r_n^d\times r_n^d$, will play an important probabilistic role related with the concept of \emph{invariant measure} associated with $\bm P$ (or $\bm{\mathcal{A}}$), as we will see now. First, we will derive a formula to directly compute $(\Pi_n)_{n\geq0}$ in terms of the coefficients $A_{n,i}$ and the generalized inverse of $C_{n+1}^T$.
\begin{lemma}
Let $(\Pi_n)_{n\geq0}$ be defined by \eqref{norms}. Then, for $n\geq1$, we have
\begin{equation}\label{formulaca}
\Pi_n=\Pi_0\sum_{i_1,\ldots,i_n\in\{1,\ldots,d\}}G_{n,i_1}G_{n-1,i_2}\cdots G_{1,i_n}A_{0,i_n}A_{1,i_{n-1}}\cdots A_{n-1,i_1},
\end{equation}
where $G_n=(G_{n,1}\,\cdots\,G_{n,d})$ is a generalized inverse of $C_n^T=(C_{n,1}\,\cdots\,C_{n,d})^T$. Moreover, the representation is independent of the choice of the generalized inverse $G_n$.
\end{lemma}
\begin{proof}
From \eqref{symm2} we have $\Pi_{n-1}A_{n-1,i}=C_{n,i}^T\Pi_{n}$. Written in terms of the joint matrices we have
\begin{equation}\label{formuddd}
C_{n}^T\Pi_{n}=\begin{pmatrix}\Pi_{n-1} &&\\&\ddots&\\&&\Pi_{n-1}\end{pmatrix}A_{n-1}.
\end{equation}
Now, multiplying on the left by a generalized inverse $G_n$ of $C_n^T$ (so that $G_nC_n^T=I$) we get
$$
\Pi_n=\sum_{i=1}^dG_{n,i}\Pi_{n-1}A_{n-1,i},\quad n\geq1.
$$
Iterating this formula we get \eqref{formulaca}. For the invariance of the representation, consider the singular-value decomposition of $C_n^T$ given by
$$
C_n^T=W_n^T\begin{bmatrix} \Lambda_n\\ \bigcirc\end{bmatrix}U_n,
$$
where $W_n, \Lambda_n$ and $U_n$ are $dr_{n-1}^d\times dr_{n-1}^d$, $r_n^d\times r_n^d$ and $r_n^d\times r_n^d$ matrices, respectively. A generalized inverse is then given by
$$
G_n=U_n^T\begin{bmatrix} \Lambda_n^{-1}&\Lambda_{n,1}\end{bmatrix}W_n,
$$
where $\Lambda_{n,1}$ is any $r_n^d\times(dr_{n-1}^d-r_n^d)$ matrix. Observe from the definition \eqref{knd} that $dr_{n-1}^d-r_n^d\geq1$ for $n,d\geq2$. $G_n$ can be written as
$$
G_n=U_n^T\begin{bmatrix} \Lambda_n^{-1}&\bigcirc\end{bmatrix}W_n+U_n^T\begin{bmatrix} \bigcirc&\Lambda_{n,1}\end{bmatrix}W_n.
$$
The first part of $G_n$ is the so-called pseudo inverse or the Moore-Penrose inverse, which is unique. Multiplying this $G_n$ on the left in \eqref{formuddd} and using again $\Pi_{n-1}A_{n-1,i}=C_{n,i}^T\Pi_{n}$, we conclude that the second part of the sum in $G_n$ must vanish, so formula \eqref{formulaca} is independent of the choice $\Lambda_{n,1}$.

\end{proof}
\begin{remark}
Observe that $\Pi_0$ in \eqref{formulaca} is a number which can be taken as 1 if we assume that the spectral measure $w(x)$ is a probability measure.
\end{remark}
\begin{remark}
Similarly, using a generalized inverse $D_n^T$ of $A_n$ (see \eqref{geninvA}) we can derive a formula for the sequence of norms $(\Pi_n^{-1})_{n\geq0}$. Indeed,
\begin{equation*}
\Pi_n^{-1}=\Pi_0^{-1}\sum_{i_1,\ldots,i_n\in\{1,\ldots,d\}}D_{n-1,i_1}^TD_{n-2,i_2}^T\cdots D_{0,i_n}^TC_{1,i_n}^TC_{2,i_{n-1}}^T\cdots C_{n,i_1}^T,
\end{equation*}
and, again, this is independent of the choice of the generalized inverse $D_n^T$.
\end{remark}
\begin{remark}
For the univariate case of birth-death chains, the matrices $\Pi_n, n\geq0,$ are now numbers, which are usually called the \emph{potential coefficients}. They can be written as
$$
\pi_0=1,\quad \pi_n=\frac{a_0\cdots a_{n-1}}{c_1\cdots c_n},\quad n\geq1,
$$
where we denote here $\pi_n=\Pi_n$, $a_n=A_{n,i}$ and $c_n=C_{n,i}$ (there is only one index $i$).
\end{remark}

\begin{theorem}\label{Teo}
Let $\bm P$ be the transition probability matrix given by \eqref{PP}. Define the sequence of matrices $\Pi_n$, $n\geq1$, as in \eqref{formulaca} with $\Pi_0=(\int_\Omega w(x)dx)^{-1}$. Consider the following row vector
\begin{equation}\label{ID}
\mbox{\boldmath$\pi$}=
(\Pi_0;(\Pi_1\bm e_{r_1})^T;(\Pi_2\bm e_{r_2})^T;\cdots),
\end{equation}
where $\bm e_N$ and $r_n$ are defined by \eqref{eeN} and \eqref{knd}, respectively.  Then $\mbox{\boldmath$\pi$}$ is an invariant measure for the discrete-time QBD process $\bm P$, i.e. all components of $\mbox{\boldmath$\pi$}$ are nonnegative and
\begin{equation}\label{IDprop}
\mbox{\boldmath$\pi$}\bm P=\mbox{\boldmath$\pi$}.
\end{equation}
\end{theorem}
\begin{proof}
From \eqref{Palpha} we can see that $\bm A_n=\sum_{i=1}^d\tau_iA_{n,i}$, $\bm B_n=\sum_{i=1}^d\tau_iB_{n,i}$ and $\bm C_n=\sum_{i=1}^d\tau_iC_{n,i}$. To prove \eqref{IDprop}, we have to check that
$$
\Pi_0\bm B_0+(\Pi_1\bm e_{r_1})^T\bm C_1=\Pi_0,
$$
and
$$
(\Pi_{n-1}\bm e_{r_{n-1}})^T\bm A_{n-1}+(\Pi_{n}\bm e_{r_n})^T\bm B_n+(\Pi_{n+1}\bm e_{r_{n+1}})^T\bm C_{n+1}=(\Pi_n\bm e_{r_n})^T,\quad n\geq1.
$$
The first equality holds using $\Pi_{0}\bm A_{0}=\bm C_{1}^T\Pi_{1}$ (see \eqref{symm2}), that $\Pi_n$ are symmetric matrices and the fact that $\bm P$ is stochastic (see \eqref{stochcond}). Therefore
$$
\Pi_0\bm B_0+\bm e_{r_1}^T\Pi_1^T\bm C_1=\Pi_0\bm B_0+\bm e_{r_1}^T\bm A_0^T\Pi_0=\Pi_0(\bm B_0\bm e_{r_0}+\bm A_0\bm e_{r_1})^T=\Pi_0.
$$
Similarly, for $n\geq1,$ and using additionally \eqref{symm1}, we get
\begin{align*}
\bm e_{r_{n-1}}^T\Pi_{n-1}\bm A_{n-1}&+\bm e_{r_n}^T\Pi_n\bm B_n+\bm e_{r_{n+1}}^T\Pi_{n+1}\bm C_{n+1}=\bm e_{r_{n-1}}^T\bm C_n^T\Pi_n+\bm e_{r_n}^T\bm B_n^T\Pi_n+\bm e_{r_{n+1}}^T\bm A_n^T\Pi_n\\
&=\left( \bm C_n\bm e_{r_{n-1}}+\bm B_n\bm e_{r_n}+\bm A_n\bm e_{r_{n+1}}\right)^T\Pi_n=\bm e_{r_n}^T\Pi_n=(\Pi_n\bm e_{r_n})^T.
\end{align*}
Also observe that by \cite[Lemma 5.6]{Se06} all components of $\mbox{\boldmath$\pi$}$ are nonnegative.
\end{proof}
\begin{remark}
The same result holds for continuous-time QBD processes, where now $\mbox{\boldmath$\pi$}$ satisfies $\mbox{\boldmath$\pi$}\bm{\mathcal{A}}=\bm 0$.
\end{remark}
\begin{remark}
The previous theorem was proved  in \cite{dI1} for QBD processes with a constant number $N$ of phases for each level, i.e. $r_n^d=N,$ for all $n\geq0$.
\end{remark}
\begin{remark}
The invariant measure $\mbox{\boldmath$\pi$}$ in \eqref{ID} will become an invariant distribution if
$$
\sum_{n=0}^\infty\sum_{j=1}^{r_n}\left(\Pi_n\bm e_{r_n}\right)_j^T<\infty.
$$
\end{remark}

Finally, let us talk about the concept of \emph{recurrence}. The definition of recurrence that we will use here is an extension of the one used in \cite{DRSZ,DR}. Consider first the case of discrete-time QBD processes. Then, using \eqref{KMcF} and Lebesgue's theorem, we have
\begin{align*}
\bm H_{i,j}(z)&=\sum_{n=0}^\infty\bm P_{i,j}^nz^n=\sum_{n=0}^\infty\left(\int_\Omega(\tau_1x_1+ \cdots+\tau_dx_d)^nz^n\mathbb{P}_i(x)\mathbb{P}_j^T(x)w(x)dx\right) \Pi_j\\
&=\left(\int_\Omega\frac{1}{1-z(\tau_1x_1+\cdots+\tau_dx_d)} \mathbb{P}_i(x)\mathbb{P}_j^T(x)w(x)dx\right)\Pi_j.
\end{align*}
Observe that each block $(i,j)$ is a matrix of dimension $r_i^d\times r_j^d$. A state $(i,l)$, where $i\in\mathbb{N}_0$ and $0\leq l\leq r_i^d$, is \emph{recurrent} if and only if
\begin{align*}
\sum_{n=0}^\infty e_l^T\bm P_{i,i}^ne_l &=\lim_{z\to1}e_l^T\bm H_{i,i}(z)e_l\\
&=e_l^T\left(\int_\Omega\frac{1}{1-(\tau_1x_1+\cdots+\tau_dx_d)} \mathbb{P}_i(x)\mathbb{P}_i^T(x)w(x)dx\right)\Pi_ie_l=\infty,
\end{align*}
for some $0\leq l\leq r_i^d$, where $e_l^T=(0,\ldots,0,1,0,\ldots,0)$ is the $l$-th canonical vector in $\mathbb{R}^{r_i^d}$. If we assume that the discrete-time QBD process is irreducible, then it is enough to study recurrence at one single state, for instance the state $(0,0)$. In this case we have $r_0^d=1, \mathbb{P}_0(x)=1$ and $\Pi_0=1$. Therefore the discrete-time QBD process is \emph{recurrent} if and only if
\begin{equation}\label{recu}
\int_\Omega\frac{w(x_1,\ldots,x_d)}{1-(\tau_1x_1+\cdots+\tau_dx_d)}dx_1\cdots dx_d=\infty.
\end{equation}
Otherwise it is transient. From the Karlin-McGregor representation \eqref{KMcF} for $i=j=0,$ it is possible to see that the discrete-time QBD process is \emph{positive recurrent} if and only if it is recurrent and the spectral weight $w$ has a jump at least at one point $x^0=(x_1^0,\ldots,x_d^0)$ such that $\tau_1x_1^0+\cdots+\tau_dx_d^0=1$.

Similar results hold for continuous-time QBD processes, but using \eqref{KMcF2} instead. Indeed, the continuous-time QBD process is \emph{recurrent} if and only if
\begin{equation}\label{recu2}
\int_\Omega\frac{w(x_1,\ldots,x_d)}{\tau_1x_1+\cdots+\tau_dx_d}dx_1\cdots dx_d=\infty,
\end{equation}
and it is \emph{positive recurrent} if and only if is recurrent and the spectral weight $w$ has a jump at least at one point $x^0=(x_1^0,\ldots,x_d^0)$ such that $\tau_1x_1^0+\cdots+\tau_dx_d^0=0$.

\section{QBD processes associated with product orthogonal polynomials}\label{SecPOP}

One simple way to generate examples of bivariate orthogonal polynomials is by considering product weight functions  $w$ of the form
$$
w(x,y)=w_1(x)w_2(y),
$$
where $w_1$ and $w_2$ are two one-variable weight functions. It is well known (see Proposition 2.2.1 in \cite{DX14}) that the bivariate polynomials defined by
$$
P_{n,k}(x,y)=p_{n-k}(x)q_k(y),\quad 0\leq k\leq n,
$$
form a mutually orthogonal basis with respect to $w$, where $(p_n)_n$ and $(q_n)_n$ are sequences of orthogonal polynomials with respect to $w_1$ and $w_2$, respectively. Observe that, in this case, $r_n^2=n+1$ where $r_n^d$ is defined by \eqref{knd}.
We will use the vector notation so we define
$$
{\mathbb P}_n(x,y)=\left( P_{n,0}(x,y),  P_{n,1}(x,y),  \dots, P_{n,n} (x,y)\right)^{T},\quad n\geq0.
$$
According to Theorem 3.3.1 in \cite{DX14}, we have that the sequence $({\mathbb P}_n)_{n\geq0}$ satisfies the following three-term recurrence relations:
\begin{equation}\label{TTRR}
\begin{aligned}
x\, {\mathbb P}_n(x,y) & = A_{n,1} {\mathbb P}_{n+1}(x,y)+ B_{n,1} {\mathbb P}_n(x,y) + C_{n,1} {\mathbb P}_{n-1}(x,y), \\
y\, {\mathbb P}_n(x,y) & = A_{n,2} {\mathbb P}_{n+1}(x,y)+ B_{n,2} {\mathbb P}_n(x,y) + C_{n,2} {\mathbb P}_{n-1}(x,y),
\end{aligned}
\end{equation}
where $A_{n,1}, A_{n,2}$ are matrices of dimension $(n+1)\times(n+2)$, $B_{n,1}, B_{n,2}$ are matrices of dimension $(n+1)\times(n+1)$, $C_{n,1}, C_{n,2}$ are matrices of dimension $(n+1)\times n$ and they satisfy the rank conditions \eqref{rankc1} and \eqref{rankc2}. From these recurrence relations we can define the \emph{block Jacobi matrices}

\begin{equation}\label{JacMat}
J_{1}=\left(\begin{array}{cccccc}
B_{0,1}   & A_{0,1}   &         &       &\bigcirc \\
C_{1,1}   &B_{1,1}   & A_{1,1} &       &          \\
          &C_{2,1}   & B_{2,1} &A_{2,1} &         \\
  \bigcirc         &           & \ddots  & \ddots & \ddots
\end{array}\right),\quad J_{2}=\left(\begin{array}{cccccc}
B_{0,2}   & A_{0,2}   &         &       &\bigcirc \\
C_{1,2}   &B_{1,2}   & A_{1,2} &       &          \\
          &C_{2,2}   & B_{2,2} &A_{2,2} &         \\
  \bigcirc         &           & \ddots  & \ddots & \ddots
\end{array}\right).
\end{equation}

If we have the three-term recurrence relations satisfied by the polynomials $(p_n)_n$ and $(q_n)_n$, i.e.
\begin{equation*}
\begin{split}
xp_n&=a_np_{n+1}+b_np_n+c_np_{n-1},\quad p_{-1}=0,\\
xq_n&=\tilde a_nq_{n+1}+\tilde b_nq_n+\tilde c_nq_{n-1},\quad q_{-1}=0,
\end{split}
\end{equation*}
then we have that the coefficients $A_{n,i},B_{n,i},C_{n,i}, i=1,2,$ in \eqref{TTRR} are given by
\begin{equation}\label{coefdiag}
\begin{split}
A_{n,1}&=\left[
\begin{array}{cccc}
a_{n} &  &  \bigcirc   & 0  \\
        & \ddots    &        &    \vdots \\
     \bigcirc   &            & a_0 &   0
     \end{array}
\right],\quad B_{n,1}=\left[
\begin{array}{ccc}
b_{n} &  &  \bigcirc    \\
        & \ddots    &    \\
     \bigcirc   &            & b_0
     \end{array}
\right],\quad C_{n,1}=\left[
\begin{array}{ccc}
c_{n} &  &  \bigcirc   \\
         & \ddots    &        \\
      \bigcirc           & &   c_1 \\
      0 &\cdots & 0
     \end{array}
\right],
\\
A_{n,2}&=\left[
\begin{array}{cccc}
0 & \tilde a_{0} & & \bigcirc    \\
    \vdots &    & \ddots    &    \\
    0& \bigcirc   &        & \tilde a_n
     \end{array}
\right],\quad B_{n,2}=\left[
\begin{array}{ccc}
\tilde b_{0} &  &  \bigcirc    \\
        & \ddots    &    \\
     \bigcirc   &            & \tilde b_n
     \end{array}
\right],\quad C_{n,2}=\left[
\begin{array}{ccc}
0 &  &  \bigcirc   \\
    \tilde c_{1}     &    &        \\
                 &\ddots & 0   \\
      \bigcirc & & \tilde c_n
     \end{array}
\right].
\end{split}
\end{equation}
These are the simplest examples since both variables are separated. Now, we will see a couple of examples related with QBD processes.

\subsection{Product Jacobi polynomials}

Let $Q_n^{(\alpha,\beta)}(x)$ be the family of Jacobi polynomials normalized in such a way that $Q_n^{(\alpha,\beta)}(1)=1$. They are orthogonal with respect to the (normalized) Jacobi weight (or Beta distribution)
$$
w(x)= \frac{\Gamma(\alpha+\beta+2)}{\Gamma(\alpha+1)\Gamma(\beta+1)}x^\alpha(1-x)^\beta,\quad x\in[0,1],\quad \alpha,\beta>-1,
$$
and they satisfy the following three-term recurrence relation
$$
xQ_n^{(\alpha,\beta)}(x)=a_n^{(\alpha,\beta)}Q_{n+1}^{(\alpha,\beta)}(x)+b_n^{(\alpha,\beta)}Q_{n}^{(\alpha,\beta)}(x)+c_n^{(\alpha,\beta)}Q_{n-1}^{(\alpha,\beta)}(x),
$$
where
\begin{equation}\label{coeffJ1}
\begin{split}
a_n^{(\alpha,\beta)}&=\frac{(n+\beta+1)(n+\alpha+\beta+1)}{(2n+\alpha+\beta+1)(2n+\alpha+\beta+2)},\\
b_n^{(\alpha,\beta)}&=1-a_n^{(\alpha,\beta)}-c_n^{(\alpha,\beta)},\\
c_n^{(\alpha,\beta)}&=\frac{n(n+\alpha)}{(2n+\alpha+\beta)(2n+\alpha+\beta+1)}.
\end{split}
\end{equation}

Let us define an inner product on the square $S=[0,1]\times[0,1]$ by
$$
\langle f,g \rangle = \frac{\Gamma(\alpha+\beta+2)\Gamma(\gamma+\delta+2)}{\Gamma(\alpha+1)\Gamma(\beta+1)\Gamma(\gamma+1)\Gamma(\delta+1)} \int_S f(x,y) g(x,y) x^\alpha (1-x)^\beta y^\gamma (1-y)^\delta dx dy,
$$
which is normalized in such a way that $\langle 1,1 \rangle=1$. For $0\le k\le n$ the set of polynomials
\begin{equation}\label{qsjac}
Q_{n,k}(x,y)=Q_{n-k}^{(\alpha,\beta)}(x) Q_k^{(\gamma,\delta)}(y),
\end{equation}
constitutes a basis of the space of orthogonal polynomials of degree $n$ with $Q_{n,k}(1,1)=1$. The vector of polynomials
${\mathbb Q}_n=\left( Q_{n,0},  Q_{n,1}, \dots, Q_{n,n} \right)^{T}$ satisfy the three-term recurrence relations
\begin{equation*}
\begin{aligned}
x\, {\mathbb Q}_n(x,y) & = A_{n,1} {\mathbb Q}_{n+1}(x,y)+ B_{n,1} {\mathbb Q}_n(x,y) + C_{n,1} {\mathbb Q}_{n-1}(x,y), \\
y\, {\mathbb Q}_n(x,y) & = A_{n,2} {\mathbb Q}_{n+1}(x,y)+ B_{n,2} {\mathbb Q}_n(x,y) + C_{n,2} {\mathbb Q}_{n-1}(x,y),
\end{aligned}
\end{equation*}
where $A_{n,i}, B_{n,i}, C_{n,i}, i=1,2,$ are given by \eqref{coefdiag} (for $a_n=a_n^{(\alpha,\beta)}, b_n=b_n^{(\alpha,\beta)}, c_n=c_n^{(\alpha,\beta)},$ and $\tilde a_n=a_n^{(\gamma,\delta)}, \tilde b_n=b_n^{(\gamma,\delta)}, \tilde c_n=c_n^{(\gamma,\delta)}$). Observe that the Jacobi matrices $J_1$ and $J_2$ are both \emph{stochastic matrices}. Now, let us consider a Jacobi matrix of the form \eqref{Palpha}, i.e.
$\bm P=\tau_1J_1+\tau_2J_2.$ Since $J_1$ and $J_2$ are both stochastic matrices, the Jacobi matrix $\bm P$ is always a stochastic matrix if and only if $\tau_2=1-\tau_1$ and $0\leq \tau_1\leq1$. For simplicity we will call $\tau=\tau_1$. Therefore,
\begin{equation*}
\bm P=\tau J_1+(1-\tau) J_2,\quad0\leq \tau\leq1,
\end{equation*}
can be regarded as the transition probability matrix of a family of discrete-time QBD processes. Thus, the Karlin-McGregor representation formula \eqref{KMcF} for the $(i,j)$ block entry of the matrix $\bm P$ is given by
\begin{equation*}\label{KMcPJ}
\bm P_{i,j}^n=C\left(\int_{S}[\tau x+(1-\tau)y]^n\mathbb{Q}_i(x,y)\mathbb{Q}_j^T(x,y) x^\alpha (1-x)^\beta y^\gamma (1-y)^\delta dxdy\right)\Pi_j,
\end{equation*}
where
$$
C=\frac{\Gamma(\alpha+\beta+2)\Gamma(\gamma+\delta+2)}{\Gamma(\alpha+1)\Gamma(\beta+1)\Gamma(\gamma+1)\Gamma(\delta+1)},
$$
and $\Pi_j$ is a diagonal matrix whose entries are given by
\begin{equation*}\label{Pisp}
\begin{split}
\Pi_{j,k}&=\frac{\sigma_{j,k}^2 }{\nu_{j,k}}, \quad k=0,1,\ldots,j, \quad\sigma_{j,k}=\frac{(\beta+1)_{j-k} (\delta+1)_k}{(j-k)!\,k!}
\\
\nu_{j,k}&=\frac{C\times\Gamma(j-k+\alpha+1)\Gamma(j-k+\beta+1)\Gamma(k+\gamma+1) \Gamma(k+\delta+1)}{(2j-2k+\alpha+\beta+1)(2k+\gamma+\delta+1)(j-k)!\, \Gamma(j-k+\alpha+\beta+1)k! \,\Gamma(k+\gamma+\delta+1)}.
\end{split}
\end{equation*}
From \eqref{KMcFd} and \eqref{qsjac} we can derive a separated expression for all probabilities, given by
\begin{align*}
\left(\bm P_{i,j}^n\right)_{i',j'}=&C\times\Pi_{j,j'}\sum_{k=0}^n\binom{n}{k}\tau^k(1-\tau)^{n-k}\left[\int_0^1Q_{i-i'}^{(\alpha,\beta)}(x)Q_{j-j'}^{(\alpha,\beta)}(x)x^{\alpha+k}(1-x)^{\beta}dx\right]\\
&\hspace{2cm}\times\left[\int_0^1Q_{i'}^{(\gamma,\delta)}(y)Q_{j'}^{(\gamma,\delta)}(y)y^{\gamma+n-k}(1-y)^{\delta}dy\right].
\end{align*}

According to Theorem \ref{Teo} we can construct an invariant measure $\bm\pi$ for the QBD process given by \eqref{ID}. The family of discrete-time QBD processes is recurrent (see \eqref{recu}) if and only if
$$
\int_S\frac{x^\alpha (1-x)^\beta y^\gamma (1-y)^\delta}{1-\tau x-(1-\tau)y}dxdy=\infty.
$$
After some computations, it turns out that, if $0<\tau<1$, this integral is divergent if and only if $\beta+\delta\leq-1$. If $\tau=1$ the divergence is equivalent to $\beta\leq0$ and if $\tau=0$ the divergence is equivalent to $\delta\leq0$. Otherwise the QBD process is transient. The QBD process can never be positive recurrent since the spectral measure is absolutely continuous and does not have any jumps. From the shape of the coefficients $A_{n,i}, B_{n,i}, C_{n,i}, i=1,2$ a diagram of the possible transitions of the QBD process generated by $\bm P$ is given in Figure 1.

\begin{figure}
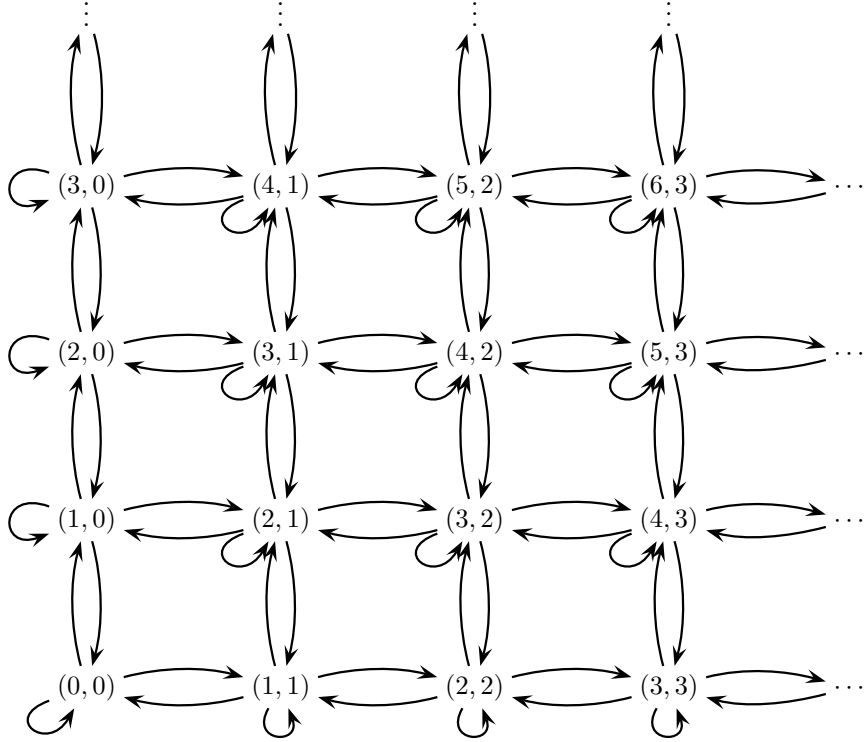

\begin{center}
$$\begin{psmatrix}[rowsep=1.8cm,colsep=1.8cm]
\rnode{21}{\vdots}&\rnode{22}{\vdots}&\rnode{23}{\vdots}&\rnode{24}{\vdots}&\\
\rnode{16}{(3,0)}&\rnode{17}{(4,1)}&\rnode{18}{(5,2)}&\rnode{19}{(6,3)}&\rnode{20}{\Huge{\cdots}}\\
\rnode{11}{(2,0)}&\rnode{12}{(3,1)}&\rnode{13}{(4,2)}&\rnode{14}{(5,3)}&\rnode{15}{\Huge{\cdots}}\\
\rnode{6}{(1,0)}&\rnode{7}{(2,1)}&\rnode{8}{(3,2)}&\rnode{9}{(4,3)}&\rnode{10}{\Huge{\cdots}}\\
\rnode{1}{(0,0)}&\rnode{2}{(1,1)}&\rnode{3}{(2,2)}&\rnode{4}{(3,3)}&\rnode{5}{\Huge{\cdots}}
\psset{nodesep=3pt,arcangle=15,labelsep=2ex,linewidth=0.3mm,arrows=->,arrowsize=1mm
3}
\nccurve[angleA=160,angleB=200,ncurv=4]{6}{6}
\nccurve[angleA=160,angleB=200,ncurv=4]{11}{11}
\nccurve[angleA=160,angleB=200,ncurv=4]{16}{16}
\nccurve[angleA=200,angleB=240,ncurv=4]{1}{1}
\nccurve[angleA=200,angleB=240,ncurv=4]{7}{7}
\nccurve[angleA=200,angleB=240,ncurv=4]{8}{8}
\nccurve[angleA=200,angleB=240,ncurv=4]{9}{9}
\nccurve[angleA=200,angleB=240,ncurv=4]{12}{12}
\nccurve[angleA=200,angleB=240,ncurv=4]{13}{13}
\nccurve[angleA=200,angleB=240,ncurv=4]{14}{14}
\nccurve[angleA=200,angleB=240,ncurv=4]{17}{17}
\nccurve[angleA=200,angleB=240,ncurv=4]{18}{18}
\nccurve[angleA=200,angleB=240,ncurv=4]{19}{19}
\nccurve[angleA=240,angleB=290,ncurv=4]{2}{2}
\nccurve[angleA=240,angleB=290,ncurv=4]{3}{3}
\nccurve[angleA=240,angleB=290,ncurv=4]{4}{4}
\ncarc{1}{2}\ncarc{2}{1}\ncarc{3}{2}\ncarc{2}{3}\ncarc{3}{4}\ncarc{4}{3}\ncarc{5}{4}\ncarc{4}{5}
\ncarc{1}{6}\ncarc{6}{1}\ncarc{2}{7}\ncarc{7}{2}\ncarc{8}{3}\ncarc{3}{8}\ncarc{9}{4}\ncarc{4}{9}
\ncarc{6}{7}\ncarc{7}{6}\ncarc{8}{7}\ncarc{7}{8}\ncarc{8}{9}\ncarc{9}{8}\ncarc{9}{10}\ncarc{10}{9}
\ncarc{6}{11}\ncarc{11}{6}\ncarc{12}{7}\ncarc{7}{12}\ncarc{8}{13}\ncarc{13}{8}\ncarc{9}{14}\ncarc{14}{9}
\ncarc{11}{12}\ncarc{12}{11}\ncarc{13}{12}\ncarc{12}{13}\ncarc{13}{14}\ncarc{14}{13}\ncarc{15}{14}\ncarc{14}{15}
\ncarc{11}{16}\ncarc{16}{11}\ncarc{17}{12}\ncarc{12}{17}\ncarc{13}{18}\ncarc{18}{13}\ncarc{19}{14}\ncarc{14}{19}
\ncarc{16}{17}\ncarc{17}{16}\ncarc{18}{17}\ncarc{17}{18}\ncarc{18}{19}\ncarc{19}{18}\ncarc{19}{20}\ncarc{20}{19}
\ncarc{16}{21}\ncarc{21}{16}\ncarc{22}{17}\ncarc{17}{22}\ncarc{18}{23}\ncarc{23}{18}\ncarc{19}{24}\ncarc{24}{19}
%\psset{labelsep=-4.2ex}\nput{90}{0}{(0,0)}
\end{psmatrix}
$$
\end{center}
\caption{Diagram of all possible transitions of the discrete-time QBD process corresponding with the product Jacobi polynomials on a square.}
\end{figure}

An interpretation of this QBD process in terms of \emph{urn models} may be stated as follows. Consider two independent urn models for the scalar Jacobi polynomials (see \cite{G3}, or more recently \cite{GdI1}). The first urn model depends on the parameters $\alpha,\beta$ and the second urn model depends on the parameters $\gamma,\delta$, where $\alpha,\beta,\gamma,\delta$ are assumed to be nonnegative integers. The parameter $\tau$ may be interpreted as the probability of heads of a (possible biased) coin which we tose before starting the QBD process. The state space of the discrete-time QBD process $\{Z_t : t=0,1,\ldots\}$ is given by all pairs $(n,k)$ where $n\in\mathbb{N}_ 0$ and $0\leq k\leq n$. The numbers $n-k$ and $k$ can be interpreted as the number of blue balls in each of the two independent urn models, being $n$ the total number of blue balls in both models. From a state $(n,k)$ there are five possible transitions between the states, except when we are in states of the form $(n,0)$ and $(n,n)$, where we only have 3 possible transitions (see Figure 1). These five transitions are given by
\begin{equation}\label{tppopj}
\begin{split}
\mathbb{P}\left[Z_1=(n+1,k+1)\; |\; Z_0=(n,k)\right]&=(1-\tau)a_k^{(\gamma,\delta)},\\
\mathbb{P}\left[Z_1=(n+1,k)\; |\; Z_0=(n,k)\right]&=\tau a_{n-k}^{(\alpha,\beta)},\\
\mathbb{P}\left[Z_1=(n-1,k)\; |\; Z_0=(n,k)\right]&=\tau c_{n-k}^{(\alpha,\beta)},\\
\mathbb{P}\left[Z_1=(n-1,k-1)\; |\; Z_0=(n,k)\right]&=(1-\tau)c_k^{(\gamma,\delta)},\\
\mathbb{P}\left[Z_1=(n,k)\; |\; Z_0=(n,k)\right]&=\tau b_{n-k}^{(\alpha,\beta)}+(1-\tau)b_k^{(\gamma,\delta)},
\end{split}
\end{equation}
where the coefficients $a_n,b_n,c_n$ are given by \eqref{coeffJ1}. This means that the increase or decrease of one blue ball at the first urn model (and no changes in the second urn model) only depends on $n-k$ (and $\alpha,\beta$). On the other hand, the increase or decrease of one blue ball at the second urn model only depends on $k$ (and $\gamma,\delta$). Therefore both components behave independently. Observe that since we are assuming that $\alpha,\beta,\gamma,\delta$ are nonnegative integers, the QBD process $\{Z_t : t=0,1,\ldots\}$ can only be (null) recurrent if and only if $\tau=1,\beta=0$ (i.e. the second urn is ignored and $\beta=0$) or $\tau=0,\delta=0$ (i.e. the first urn is ignored and $\delta=0$). Otherwise, the QBD process is transient.

\begin{remark}
Observe that we could have relabeled the states in the form $(h,k), h,k\in\mathbb{N}_0$, where $h=n-k$, and in this way it is more clear that the transitions act independently on both components. This relabeling of the states keeps the same transitions in the diagram in Figure 1, but it will considerably change in the examples of orthogonal polynomials on a parabolic domain and on the triangle in Sections \ref{secpar} and \ref{sectri}, respectively.
\end{remark}

\begin{remark}
In the previous situation we have normalized the polynomials at the upper right corner $(1,1)$ of the square $S=[0,1]\times[0,1]$ in such a way that $Q_{n,k}(1,1)=1$. It is possible to see that we can also get probabilistic interpretations of this example if we normalize the polynomials at any corner of the unit square. For instance, if we choose to normalize the polynomials in such a way that $Q_{n,k}(0,0)=1$, then we will obtain a two-parameter family of continuous-time QBD processes with infinitesimal generators $\bm{\mathcal{A}}=\tau_1J_1+\tau_2J_2$, with $\tau_1,\tau_2\geq0$ (observe that the coefficients in $J_1$ and $J_2$ will change after the normalization of the polynomials). The same can be done for the points $(1,0)$ and $(0,1)$ where now we will get one-parameter families of discrete-time QBD processes where the free parameter will depend on the values of $\alpha,\beta,\gamma,\delta$. We will see a similar situation later in Section \ref{sectri}.
\end{remark}

\subsection{Product Laguerre polynomials}\label{secplpl}

Let $L_n^{(\alpha)}(x)$ be the classical family of Laguerre polynomials normalized in such a way that
$$
L_n^{(\alpha)}(0)=\binom{n+\alpha}{n}.
$$
They are orthogonal with respect to the (normalized) Laguerre weight (or Gamma distribution)
$$
w(x)= \frac{1}{\Gamma(\alpha+1)}x^\alpha e^{-x},\quad x\in[0,\infty),\quad \alpha>-1,
$$
and they satisfy the following three-term recurrence relation
$$
-xL_n^{(\alpha)}(x)=a_n^{(\alpha)}L_{n+1}^{(\alpha)}(x)+b_n^{(\alpha)}L_{n}^{(\alpha)}(x)+c_n^{(\alpha)}L_{n-1}^{(\alpha)}(x),
$$
where
\begin{equation}\label{coeffL1}
a_n^{(\alpha)}=n+1,\quad b_n^{(\alpha)}=-(2n+\alpha+1),\quad c_n^{(\alpha)}=n+\alpha.
\end{equation}
Let us define an inner product on the first quadrant $\mathcal{C}=[0,\infty)\times[0,\infty)$ by
$$
\langle f,g \rangle = \frac{1}{\Gamma(\alpha+1)\Gamma(\beta+1)} \int_\mathcal{C} f(x,y) g(x,y) x^\alpha y^\beta e^{-x-y} dx dy,
$$
which is normalized in such a way that $\langle 1,1 \rangle=1$. For $0\le k\le n$ the set of polynomials
\begin{equation}\label{plp}
Q_{n,k}(x,y)=L_{n-k}^{(\alpha)}(x) L_k^{(\beta)}(y),
\end{equation}
constitutes a basis of the space of orthogonal polynomials of degree $n$ with
$$
Q_{n,k}(0,0)=\binom{n-k+\alpha}{\alpha}\binom{k+\beta}{k}.
$$
The vector of polynomials
${\mathbb Q}_n=\left( Q_{n,0}, Q_{n,1}, \dots,  Q_{n,n} \right)^{T}$ satisfy the three-term recurrence relations
\begin{equation*}
\begin{aligned}
-x\, {\mathbb Q}_n(x,y) & = A_{n,1} {\mathbb Q}_{n+1}(x,y)+ B_{n,1} {\mathbb Q}_n(x,y) + C_{n,1} {\mathbb Q}_{n-1}(x,y), \\
-y\, {\mathbb Q}_n(x,y) & = A_{n,2} {\mathbb Q}_{n+1}(x,y)+ B_{n,2} {\mathbb Q}_n(x,y) + C_{n,2} {\mathbb Q}_{n-1}(x,y),
\end{aligned}
\end{equation*}
where $A_{n,i}, B_{n,i}, C_{n,i}, i=1,2,$ are given by \eqref{coefdiag} (for $a_n=a_n^{(\alpha)}, b_n=b_n^{(\alpha)}, c_n=c_n^{(\alpha)},$ and $\tilde a_n=a_n^{(\beta)}, \tilde b_n=b_n^{(\beta)}, \tilde c_n=c_n^{(\beta)}$). Observe that the Jacobi matrices $J_1$ and $J_2$ are both the (\emph{nonconservative}) infinitesimal operator of a continuous-time (diagonal) QBD process. Now, let us consider a Jacobi matrix of the form \eqref{Palpha}, i.e.
$\bm{\mathcal{A}}=\tau_1J_1+\tau_2J_2.$ The Jacobi matrix $\bm{\mathcal{A}}$ is always the infinitesimal operator of a continuous-time QBD process if and only if $\tau_1, \tau_2\geq0$. Thus, the Karlin-McGregor representation formula \eqref{KMcF2} for the $(i,j)$ block entry of the transition function matrix $\bm P(t)$ is given by
\begin{equation*}\label{KMcPL}
\bm P_{i,j}(t)=\frac{1}{\Gamma(\alpha+1)\Gamma(\beta+1)}\left(\int_\mathcal{C} e^{-(\tau_1x+\tau_2y)t}\mathbb{Q}_i(x,y)\mathbb{Q}_j^T(x,y)x^{\alpha} y^{\beta} e^{-x-y}dxdy\right)\Pi_j,
\end{equation*}
where
$\Pi_j$ is a diagonal matrix whose entries are given by
\begin{equation*}\label{Pisll}
\Pi_{j,k}=\frac{\Gamma(\alpha+1)\Gamma(\beta+1)(j-k)! \, k!}{\Gamma(j-k+\alpha+1) \Gamma(k+\beta+1)}, \qquad k=0,1,\ldots,j.
\end{equation*}
As before, from \eqref{KMcFd2} and \eqref{plp} we can derive a separated expression for all probabilities, given by
\begin{align*}
\left(\bm P_{i,j}(t)\right)_{i',j'}=&\Pi_{j,j'}\left[\int_0^\infty e^{-\tau_1xt}L_{i-i'}^{(\alpha)}(x)L_{j-j'}^{(\alpha)}(x)x^{\alpha}e^{-x}dx\right]\left[\int_0^\infty e^{-\tau_2yt}L_{i'}^{(\beta)}(y)L_{j'}^{(\beta)}(y)y^{\beta}e^{-y}dy\right].
\end{align*}

According to Theorem \ref{Teo} we can construct an invariant measure $\bm\pi$ for the QBD process given by \eqref{ID}. Finally, the family of continuous-time QBD processes is recurrent (see \eqref{recu2}) if and only if
$$
\int_\mathcal{C}\frac{x^\alpha y^\beta e^{-x-y}}{\tau_1x+\tau_2y}dxdy=\infty.
$$
After some computations, it turns out that, if $\tau_1,\tau_2>0$, this integral is divergent if and only if $\alpha+\beta\leq-1$. If $\tau_1=0$ the divergence is equivalent to $\beta\leq0$ and if $\tau_2=0$ the divergence is equivalent to $\alpha\leq0$. Otherwise the QBD process is transient. Again, the QBD process can never be positive recurrent since the spectral measure is absolutely continuous and does not have any jumps. A diagram of the possible transitions of the QBD process generated by $\bm{\mathcal{A}}$ is similar to the one given in Figure 1, but without self-transitions.

An interpretation of this QBD process is similar to the situation considered in the previous case of product Jacobi polynomials, but changing the urn models by two independent \emph{linear growth models}, similar to the models studied in \cite{KMc7, Mil}), but in these papers both components are dependent of each other. The parameters $\tau_1,\tau_2\geq0$ may be interpreted as an initial preference of choosing either one of these linear growth models. Again, the state space of the continuous-time QBD process $\{Z_t : t\geq0\}$ is given by all pairs $(n,k)$ where $n\in\mathbb{N}_ 0$ and $0\leq k\leq n$ and now $n-k$ and $k$ can be interpreted as the number of elements in the population in each of the models. From a state $(n,k)$ there are four possible transitions between the states, except when we are in states of the form $(n,0)$ and $(n,n)$, where we only have 2 possible transitions (see Figure 1). As in \eqref{tppopj}, during an interval $(t,t+h)$ of infinitesimal length $h>0$, the infinitesimal birth and death rates of the process are given by
\begin{equation*}
\begin{split}
\mathbb{P}\left[Z_{t+h}=(n+1,k+1)\; |\; Z_t=(n,k)\right]&=\tau_2a_k^{(\beta)}h+o(h),\\
\mathbb{P}\left[Z_{t+h}=(n+1,k)\; |\; Z_t=(n,k)\right]&=\tau_1a_{n-k}^{(\alpha)}h+o(h),\\
\mathbb{P}\left[Z_{t+h}=(n-1,k)\; |\; Z_t=(n,k)\right]&=\tau_1c_{n-k}^{(\alpha)}h+o(h),\\
\mathbb{P}\left[Z_{t+h}=(n-1,k-1)\; |\; Z_t=(n,k)\right]&=\tau_2c_k^{(\beta)}h+o(h),
\end{split}
\end{equation*}
where the coefficients $a_n,b_n,c_n$ are given by \eqref{coeffL1}. Again we can see from the coefficients that both components behave independently. Now, if we assume that we do not ignore any of the populations (i.e. $\tau_1,\tau_2>0$), it is possible that the QBD process is  (null) recurrent if we choose negative $\alpha,\beta$ such that $\alpha+\beta\leq-1$. Another important observation now is that there is a positive probability that the QBD process is killed if the process is located at one of the states of the form $(n,0)$ or $(n,n)$ for $n\geq0$ (i.e. the boundary in the grid in Figure 1).

\begin{remark}
Observe that we could have taken another normalization of the polynomials $Q_{n,k}(x,y)$ in \eqref{plp} in such a way that $Q_{n,k}(0,0)=1$. In that situation, the coefficients of the three-term recurrence relation for the new normalized Laguerre polynomials are $a_n^{(\alpha)}=n+\alpha+1, b_n^{(\alpha)}=-(2n+\alpha+1), c_n^{(\alpha)}=n$. Then, we will obtain another family of continuous-time QBD processes. The only difference is that this model is \emph{conservative}, meaning that the process will evolve always in time and will never stop, unlike the case we studied before.
\end{remark}

\begin{remark}\label{remJL}
It is also possible to consider product Jacobi-Laguerre polynomials, in which case the region is given by the strip $\mathcal{S}=[0,1]\times[0,\infty)$. After a proper normalization of the polynomials, it is possible to see that the only corner for which we get a probabilistic interpretation of this example is $(0,0)$, but not $(0,1)$. This is due to the fact that the coefficients of the three-term recurrence relation for the Laguerre polynomials are unbounded, contrary to the coefficients for the Jacobi polynomials.
\end{remark}

\section{QBD processes associated with orthogonal polynomials on a parabolic domain}\label{secpar}

In \cite{Ko75}, T. Koornwinder studied analogues of Jacobi orthogonal polynomials in two variables. In particular, he established seven different classes of bivariate orthogonal polynomials, some of them obtained by using a construction defined by Agahanov in \cite{Ag65}. One of these classes are orthogonal polynomials on the domain
$$
R=\{(x,y)\in \mathbb{R}^2: y^2 < x < 1 \},
$$
bounded by a straight line and a parabola. For $ \alpha, \beta >- 1,$ the inner product is given by the integral
$$
   \langle f,g \rangle =\frac{\Gamma(\alpha+\beta+\frac{5}{2})}{\sqrt{\pi}\, \Gamma(\alpha+1)\Gamma(\beta+1)} \int_{R}\, f(x,y)\, g(x,y) \, (1-x)^{\alpha}(x-y^2)^{\beta} \, dx\, dy,
$$
where the weight is normalized in such a way that $ \langle 1,1 \rangle=1$. A mutually orthogonal basis of polynomials $\{P_{n,k} : 0\leq k\leq n \}$ can be obtained from a modified product of Jacobi polynomials in this way
\begin{equation}\label{qsjacp}
P_{n,k}(x,y)=P_{n-k}^{(\alpha,\beta+k+1/2)}(2x-1)\, x^{k/2}\,P_k^{(\beta,\beta)}\left(\frac{y}{\sqrt{x}} \right).
\end{equation}
Here $P_n^{(\alpha,\beta)}(t)$ are the standard Jacobi polynomials (see \cite[Chapter 22]{AS72} or \cite{Sz78}). For $\alpha, \beta > -1$, the Jacobi
polynomials are orthogonal with respect to the weight function
\begin{equation}\label{pesoJ}
   w_{\alpha,\beta}(t) = (1-t)^\alpha(1+t)^{\beta}, \qquad  -1< t < 1,
\end{equation}
and they satisfy the properties
\begin{equation} \label{jac-norm}
P_n^{(\alpha,\beta)}(1) = {n+\alpha\choose n}= \frac{(\alpha+1)_n}{n!},\quad P_n^{(\alpha,\beta)}(-1) = (-1)^n {n+\beta\choose n} = (-1)^n \frac{(\beta+1)_n}{n!}.
\end{equation}

We look for a basis of polynomials $\{Q_{n,k} : 0\leq k\leq n \}$ satisfying $Q_{n,k}(1,1)=1$ so, if we denote
\begin{equation}\label{sigpar}
\sigma_{n,k}=P_{n,k}(1,1)=P_{n-k}^{(\alpha,\beta+k+1/2)}(1)\, P_k^{(\beta,\beta)}(1)=\frac{(\alpha +1)_{n-k}}{(n-k)!} \frac{(\beta+1)_{k}}{k!},
\end{equation}
and we define $Q_{n,k}(x,y)=\sigma_{n,k}^{-1}P_{n,k}(x,y),$ the condition holds. We can use vector notation and the vector polynomials ${\mathbb Q}_n=\left( Q_{n,0},  Q_{n,1}, \dots,  Q_{n,n} \right)^{T}
$ satisfy the three-term recurrence relations
\begin{equation*}\label{TTRRpar}
\begin{aligned}
x\, {\mathbb Q}_n(x,y) & = A_{n,1} {\mathbb Q}_{n+1}(x,y)+ B_{n,1} {\mathbb Q}_n(x,y) + C_{n,1} {\mathbb Q}_{n-1}(x,y), \\
y\, {\mathbb Q}_n(x,y) & = A_{n,2} {\mathbb Q}_{n+1}(x,y)+ B_{n,2} {\mathbb Q}_n(x,y) + C_{n,2} {\mathbb Q}_{n-1}(x,y),
\end{aligned}
\end{equation*}
where the Jacobi matrices have a special shape (see \cite{DX14, MPP17}). On one side, the matrices $A_{n,1}$, $B_{n,1}$ and $C_{n,1}$ are diagonal matrices:
\begin{equation}\label{ABC1}
\begin{array}{c}
A_{n,1}=\left[
\begin{array}{ccccccc}
a_{n,0} &            &        &             & 0  \\
        & a_{n,1}    &        &             & \vdots \\
        &            & \ddots &             &         \\
        &            &        & a_{n,n}   & 0
\end{array}
\right],
\\[1cm]
B_{n,1}=\left[
\begin{array}{ccccccc}
b_{n,0}     &            &        &          \\
            & b_{n,1}    &        &           \\
            &            & \ddots &        \\
            &            &        & b_{n,n}
\end{array}
\right],
\quad
C_{n,1}=\left[
\begin{array}{cccccccc}
c_{n,0}     &            &        &               \\
            & c_{n,1}    &        &               \\
            &            & \ddots &             \\
            &            &        & c_{n,n-1} \\
     0      &            & \dots   &  0
\end{array}
\right].
\end{array}
\end{equation}
On the other side, the matrices $A_{n,2}$, $B_{n,2}$ and $C_{n,2}$ are tridiagonal matrices:

\begin{equation}\label{ABC2}
\begin{array}{c}
A_{n,2}=\left[
\begin{array}{ccccccc}
a_{n,0}^{(2)} & a_{n,0}^{(3)}&         &          &  \\
a_{n,1}^{(1)} & a_{n,1}^{(2)}& a_{n,1}^{(3)} &         &    \\
        & \ddots & \ddots  & \ddots  &         \\
 &        & a_{n,n}^{(1)} & a_{n,n}^{(2)} & a_{n,n}^{(3)}
\end{array}
\right],
\quad
B_{n,2}=\left[
\begin{array}{ccccccc}
b_{n,0}^{(2)} & b_{n,0}^{(3)}&         &        &  \\
b_{n,1}^{(1)} & b_{n,1}^{(2)}& b_{n,1}^{(3)} &        &           \\
              & \ddots & \ddots  & \ddots &           \\
              &        & b_{n,n-1}^{(1)} &  b_{n,n-1}^{(2)} & b_{n,n-1}^{(3)}  \\
              &        &       &  b_{n,n}^{(1)}     & b_{n,n}^{(2)}
\end{array}
\right],
\\[1cm]
C_{n,2}=\left[
\begin{array}{ccccccc}
c_{n,0}^{(2)} & c_{n,0}^{(3)}&         &          &  \\
c_{n,1}^{(1)} & c_{n,1}^{(2)}& c_{n,1}^{(3)} &         &    \\
        & \ddots & \ddots  & \ddots  &         \\
         &        & c_{n,n-2}^{(1)} & c_{n,n-2}^{(2)} & c_{n,n-2}^{(3)} \\
         &        &        &  c_{n,n-1}^{(1)} & c_{n,n-1}^{(2)} \\
        &        &        &         & c_{n,n}^{(1)}
\end{array}
\right].
\end{array}
\end{equation}
The elements in the coefficients $A_{n,1}$, $B_{n,1}$ and $C_{n,1}$ are given by
$$
\begin{aligned}
a_{n,k}&=\frac{(n-k+\alpha+1)(n+\alpha+\beta+3/2)}{(2n-k+\alpha+\beta+3/2)_2},\quad k=0,1,\ldots,n, \\
b_{n,k}&=\frac{(n-k+1)(n-k+\alpha+1)}{(2n-k+\alpha+\beta+3/2)_2} +\frac{(n+\alpha+\beta+1/2)(n+\beta+1/2)}{(2n-k+\alpha+\beta+1/2)_2},\quad k=0,1,\ldots,n,
\\
c_{n,k}&=\frac{(n-k)(n+\beta+1/2)}{(2n-k+\alpha+\beta+1/2)_2},\quad k=0,1,\ldots,n-1,
\end{aligned}
$$
while the elements in the coefficients $A_{n,2}$, $B_{n,2}$ and $C_{n,2}$ are given by
\begin{equation}\label{coeffpar}
\begin{split}
a_{n,k}^{(1)}&=0,\quad k=1,\ldots,n,\quad a_{n,k}^{(2)}=0,\quad  k=0,1,\ldots,n,\\
a_{n,k}^{(3)}&= \frac{(k+2\beta+1)(n+\alpha+\beta+3/2)}{(2k+2\beta+1)\, (2n-k+\alpha+\beta+3/2)},\quad k=0,1,\ldots,n,
\\
b_{n,k}^{(1)}&=\frac{k(n-k+\alpha+1)}{(2k+2\beta+1)(2n-k+\alpha+\beta+3/2)},\quad k=1,\ldots,n,
\\
b_{n,k}^{(2)}&=0,\quad k=0,1,\ldots,n,
\\
b_{n,k}^{(3)}&=\frac{(k+2\beta+1)(n-k)}{ (2k+2\beta+1)\, (2n-k+\alpha+\beta+3/2)},\quad k=0,1,\ldots,n-1,
\\
c_{n,k}^{(1)}&=\frac{k(n+\beta+1/2)}{(2k+2\beta+1)(2n-k+\alpha+\beta+3/2)},\quad k=1,\ldots,n,
\\
c_{n,k}^{(2)}&=0,\quad k=0,1,\ldots,n-1,\quad c_{n,k}^{(3)}=0,\quad  k=0,1,\ldots,n-2.
\end{split}
\end{equation}
Similar results hold when we normalize the polynomials $Q_{n,k}$ at the point $(1,-1)$. The only change is to multiply $\sigma_{n,k}$ in \eqref{sigpar} by $(-1)^k$. Observe that the previous coefficients are not separable in the variables $n$ and $k$, unlike the case of product orthogonal polynomials.

It is possible to see that the Jacobi matrices $J_1$ and $J_2$ in \eqref{JacMat} are indeed both stochastic matrices. Therefore, we get discrete-time QBD processes (the first one being trivial). For instance, a diagram of the possible transitions of the QBD process generated by $J_2$ is given in Figure 2.
%\vspace{1cm}
\begin{figure}
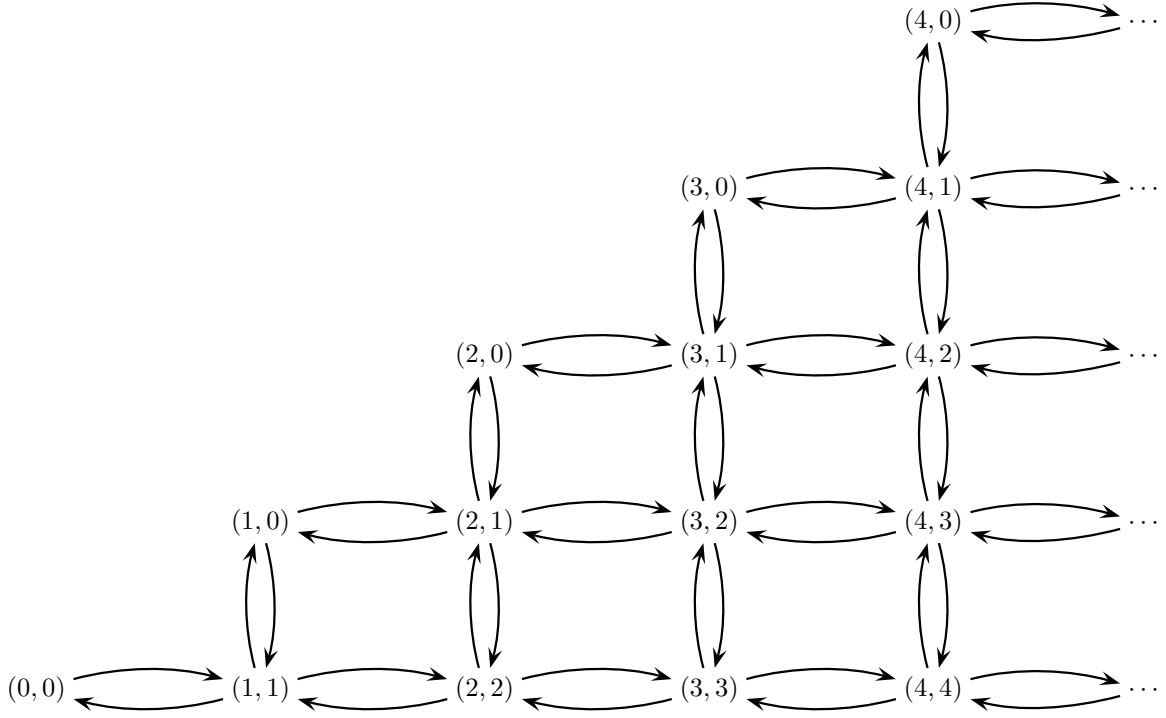

\begin{center}
$$\begin{psmatrix}[rowsep=1.8cm,colsep=2.2cm]
  &  && & \rnode{10}{(4,0)}& \rnode{16}{\Huge{\cdots}} \\
    &  & &\rnode{6}{(3,0)}& \rnode{11}{(4,1)}& \rnode{17}{\Huge{\cdots}} \\
     & & \rnode{3}{(2,0)} &\rnode{7}{(3,1)}& \rnode{12}{(4,2)}& \rnode{18}{\Huge{\cdots}} \\
    &\rnode{1}{(1,0)} &\rnode{4}{(2,1)} &\rnode{8}{(3,2)} & \rnode{13}{(4,3)}& \rnode{19}{\Huge{\cdots}} \\
  \rnode{0}{(0,0)}&\rnode{2}{(1,1)}&\rnode{5}{(2,2)}&\rnode{9}{(3,3)} & \rnode{14}{(4,4)}& \rnode{20}{\Huge{\cdots}}
\psset{nodesep=3pt,arcangle=15,labelsep=2ex,linewidth=0.3mm,arrows=->,arrowsize=1mm
3}
\ncarc{2}{0}\ncarc{0}{2} \ncarc{1}{2}\ncarc{2}{1}
\ncarc{1}{4}\ncarc{4}{1}\ncarc{3}{4}\ncarc{4}{3}
\ncarc{4}{5}\ncarc{5}{4}\ncarc{4}{8}\ncarc{8}{4}\ncarc{3}{7}\ncarc{7}{3}\ncarc{6}{7}\ncarc{7}{6}
\ncarc{7}{8}\ncarc{8}{7}\ncarc{9}{8}\ncarc{8}{9}\ncarc{6}{11}\ncarc{11}{6}\ncarc{7}{12}\ncarc{12}{7}\ncarc{8}{13}\ncarc{13}{8}\ncarc{10}{11}\ncarc{11}{10}\ncarc{12}{11}\ncarc{11}{12}\ncarc{13}{12}\ncarc{12}{13}\ncarc{13}{14}\ncarc{14}{13}\ncarc{10}{16}\ncarc{16}{10}\ncarc{11}{17}\ncarc{17}{11}\ncarc{12}{18}\ncarc{18}{12}\ncarc{13}{19}\ncarc{19}{13}
\ncarc{2}{5}\ncarc{5}{2}\ncarc{9}{5}\ncarc{5}{9}\ncarc{9}{14}\ncarc{14}{9}\ncarc{20}{14}\ncarc{14}{20}
%\psset{labelsep=-4.2ex}\nput{90}{0}{(0,0)}
\end{psmatrix}
$$
\end{center}
\caption{Diagram of all possible transitions of the discrete-time QBD process corresponding with $J_2$ for the orthogonal polynomials on a parabolic domain.}
\end{figure}

Now consider a Jacobi matrix of the form \eqref{Palpha}, i.e.
\begin{equation*}
\bm P=\tau_1 J_1+\tau_2 J_2.
\end{equation*}
Since $J_1$ and $J_2$ are both stochastic matrices, the Jacobi matrix $\bm P$ is a stochastic matrix if and only if $\tau_2=1-\tau_1$ and $0\leq \tau_1\leq1$. For simplicity, we will call $\tau=\tau_1$. Therefore
\begin{equation}\label{QBDpd}
\bm P=\tau J_1+(1-\tau) J_2,\quad0\leq \tau\leq1,
\end{equation}
is always a family of discrete-time QBD processes. Thus, the Karlin-McGregor representation formula \eqref{KMcF} for the $(i,j)$ block entry of the matrix $\bm P$ is given by
\begin{equation*}
\bm P_{i,j}^n=C\left(\int_{R}[\tau x+(1-\tau)y]^n\mathbb{Q}_i(x,y)\mathbb{Q}_j^T(x,y)(1-x)^\alpha(x-y^2)^\beta dxdy\right)\Pi_j,
\end{equation*}
where
$$
C=\frac{\Gamma(\alpha+\beta+\frac{5}{2})}{\sqrt{\pi}\, \Gamma(\alpha+1)\Gamma(\beta+1)},
$$
and $\Pi_j$ is a diagonal matrix whose entries can be computed using \eqref{formulaca} (for $\Pi_0=1$). Indeed, we have, for $k=0,1,\ldots,j,$
\begin{equation*}
\Pi_{j,k}=\frac{\sqrt{\pi}(2k+2\beta+1)(2j-k+\alpha+\beta+3/2)\Gamma(j+\alpha+\beta+3/2)\Gamma(k+2\beta+1)\Gamma(j-k+\alpha+1)}{2^{2\beta+1}\Gamma(j+\beta+3/2)\Gamma(\alpha+\beta+5/2)\Gamma(\alpha+1)\Gamma(\beta+1)(j-k)!k!}.
\end{equation*}
From \eqref{KMcFd} and \eqref{qsjacp} we can derive a separated expression for all probabilities, given by
\begin{align*}
\left(\bm P_{i,j}^n\right)_{i',j'}=&\frac{C\times\Pi_{j,j'}}{\sigma_{i,i'}\sigma_{j,j'}}\sum_{k=0}^n\binom{n}{k}\tau^k(1-\tau)^{n-k}\\
&\times\left(\int_R x^{k+i'/2+j'/2}y^{n-k}P_{i-i'}^{(\alpha,\beta+i'+1)}(2x-1)P_{j-j'}^{(\alpha,\beta+j'+1)}(2x-1)\right.\\
&\hspace{2cm}\times \left.P_{i'}^{(\beta,\beta)}\left(\frac{x}{\sqrt{y}}\right)P_{j'}^{(\beta,\beta)}\left(\frac{x}{\sqrt{y}}\right)(1-x)^\alpha(x-y^2)^\beta dxdy\right).
\end{align*}

According to Theorem \ref{Teo} we can construct an invariant measure $\bm\pi$ for the QBD process given by \eqref{ID}. Finally, the family of discrete-time QBD processes is recurrent (see \eqref{recu}) if and only if
$$
\int_R\frac{(1-x)^\alpha(x-y^2)^\beta}{1-\tau x-(1-\tau)y}dxdy=\infty.
$$
After some computations, it turns out that, if $0\leq \tau<1$, this integral is divergent if and only if $\alpha+\beta\leq-1$. If $\tau=1$ the divergence is equivalent to $\alpha\leq0$. Otherwise the QBD process is transient. Again, the QBD process can never be positive recurrent since the spectral measure is absolutely continuous and does not have any jumps.

\subsection{An urn model for the orthogonal polynomials on a parabolic domain}

In this section we will give a probabilistic interpretation of one of the QBD models introduced in the previous subsection. For simplicity, we will study the case of the discrete-time QBD process \eqref{QBDpd} with $\tau=0$, so that $\bm P=J_2$ (see Figure 2). Assume that $\alpha,\beta$ are nonnegative integers. Consider $\{Z_t : t=0,1,\ldots\}$ the discrete-time QBD process on the state space $\{(n,k) : 0\leq k\leq n, n\in\mathbb{N}_0\}$ whose one-step transition probability matrix is given by $\bm P=J_2$ (see  \eqref{ABC2} and \eqref{coeffpar}). We have two urns A and B and, at every time step $t=0,1,2,\ldots$, the state $(n,k)$ will represent the number of $n$ blue balls in urn A and the number of $k$ blue balls in urn B. Now, in urn B we add/remove red balls until we have $k+2\beta+1$ and draw one ball from the urn at random with the uniform distribution. We have two possibilities:
\begin{enumerate}
\item If we get a blue ball then we add/remove balls in urn A until we have $2n-2k+2\alpha+2$ blue balls and $2n+2\beta+1$ red balls. Then we draw again one ball from urn A and we have two possibilities:
\begin{itemize}
\item If we get a blue ball then we leave urn A with $n$ blue balls and urn B with $k-1$ blue balls and start over. Therefore, joining both steps, we have
$$
\mathbb{P}\left[Z_1=(n,k-1)\; |\; Z_0=(n,k)\right]=\frac{k}{2\beta+2k+1}\frac{2(n-k+\alpha+1)}{4n-2k+2\alpha+2\beta+3}.
$$
Observe that this probability is given by $b_{n,k}^{(1)}$ in \eqref{coeffpar}.
\item If we get a red ball then we leave urn A with $n-1$ blue balls and urn B with $k-1$ blue balls and start over. Therefore, joining both steps, we have
$$
\mathbb{P}\left[Z_1=(n-1,k-1)\; |\; Z_0=(n,k)\right]=\frac{k}{2\beta+2k+1}\frac{2n+2\beta+1}{4n-2k+2\alpha+2\beta+3}.
$$
Observe that this probability is given by $c_{n,k}^{(1)}$ in \eqref{coeffpar}.
\end{itemize}
\item If we get a red ball then we add/remove balls in urn A until we have $2n+2\alpha+2\beta+3$ blue balls and $2n-2k$ red balls. Then we draw again one ball from urn A and we have two possibilities:
\begin{itemize}
\item If we get a blue ball then we leave urn A with $n+1$ blue balls and urn B with $k+1$ blue balls and start over. Therefore, joining both steps, we have
$$
\mathbb{P}\left[Z_1=(n+1,k+1)\; |\; Z_0=(n,k)\right]=\frac{k+2\beta+1}{2\beta+2k+1}\frac{2n+2\alpha+2\beta+3}{4n-2k+2\alpha+2\beta+3}.
$$
Observe that this probability is given by $a_{n,k}^{(3)}$ in \eqref{coeffpar}.
\item If we get a red ball then we leave urn A with $n$ blue balls and urn B with $k+1$ blue balls and start over. Therefore, joining both steps, we have
$$
\mathbb{P}\left[Z_1=(n,k+1)\; |\; Z_0=(n,k)\right]=\frac{k+2\beta+1}{2\beta+2k+1}\frac{2n-2k}{4n-2k+2\alpha+2\beta+3}.
$$
Observe that this probability is given by $b_{n,k}^{(3)}$ in \eqref{coeffpar}.
\end{itemize}
\end{enumerate}
Therefore from a state $(n,k)$ there are four possible transitions between the states, except when the number of blue balls in urn B is zero, i.e. the state $(n,0)$, in which case we only have two transitions, or when the initial state is $(0,0)$ where the only transition is to $(1,0)$ with probability 1 (see Figure 2). Since we are assuming that $\alpha$ and $\beta$ are nonnegative integers, this urn model will always be a transient process.

\section{QBD processes associated with orthogonal polynomials on the triangle}\label{sectri}

Orthogonal polynomials on the triangle were first introduced by Proriol in \cite{Pr57} and after that they have been studied by several authors. The classical inner product on the triangle
$$
\mathbb{T}^2 :=\{(x,y)\in \mathbb{R}^2: x+y \leq 1, \, x\geq 0, \,y\geq 0\}
$$
is given by
$$
   \langle f,g \rangle = \frac{\Gamma(\alpha+\beta+\gamma+3)}{\Gamma(\alpha+1) \Gamma(\beta+1)\Gamma(\gamma+1)} \int_{{\mathbb{T}^2}}\, f(x,y)\, g(x,y) \, x^{\alpha} y^{\beta} (1-x-y)^{\gamma} \, dx,  \quad \alpha, \beta, \gamma >- 1.
$$
Some results about this inner product and different bases of orthogonal polynomials with respect to it can be found in \cite[pp. 35]{DX14}. For $0\leq k \leq n$, we can define
\begin{equation}\label{poltri}
P_{n,k}(x,y)= P_{n-k}^{(2k+\beta+\gamma+1,\alpha)}(2x-1) \, (1-x)^{k} \, P_k^{(\gamma,\beta)}\left(\frac{2y}{1-x}-1 \right),
\end{equation}
where $P_n^{(\alpha,\beta)}(t)$ is the standard Jacobi polynomial orthogonal with respect to the weight function \eqref{pesoJ}. Then $\{P_{n,k}\, : \, 0\leq k \leq n \}$ is a basis of the space of orthogonal polynomials of degree $n$, and the square norm of the polynomial $P_{n,k}$, denoted by $\nu_{n,k}$, is given by
\begin{equation*}
\nu_{n,k}= \frac{(n+k+\alpha +\beta +\gamma +2)(k+\beta +\gamma +1)(\alpha +1)_{n-k}\,(\beta+1)_k\, (\gamma +1)_k\, (\beta +\gamma +2)_{n+k}}{(n-k)! \, k!\,(2n+\alpha +\beta +\gamma +2)(2k+\beta +\gamma +1)(\beta +\gamma +2)_k \, (\alpha +\beta +\gamma +3)_{n+k}}.
\end{equation*}
These polynomials satisfy the three term recurrence relations \eqref{TTRR} and the matrix coefficients of these relations are of the same form as in \eqref{ABC1} and \eqref{ABC2}. Now we will normalize the polynomials in such a way that all of them equal 1 at one of the boundary points of the support of the measure. The boundary in this case is formed by all points in the border of the triangle, but it turns out that not all of these boundary points lead to a probabilistic model. We have found that normalizing at the vertices $(0,1)$ and $(0,0)$ gives coefficients of the three term recurrence relation with probabilistic interpretations. The problem with the vertex $(1,0)$ is that $P_{n,k}(1,0)=0$ for $k=1,\ldots,n$ (see \eqref{poltri}) so it is not possible to normalize the way we are looking for.

\subsection{Normalization at the point $(0,1)$}\label{secnorm1}
Using \eqref{jac-norm}, let us denote
$$
\sigma_{n,k}=P_{n,k}(0,1)= P_{n-k}^{(2k+\beta+\gamma+1,\alpha)}(-1) \, P_k^{(\gamma,\beta)}(1)= (-1)^{n-k}  \frac{(\alpha+1)_{n-k}}{(n-k)!} \frac{(\gamma+1)_k}{k!},
$$
and let us define the polynomials $Q_{n,k}$ by $ Q_{n,k}(x,y)=\sigma_{n,k}^{-1} P_{n,k}(x,y).$
This new basis of orthogonal polynomials $\{Q_{n,k},0\leq k\leq n\}$ satisfies $Q_{n,k}(0,1)=1$ for all $n\geq0$ and $0\leq k\leq n$. The inverse of the square norms $\Pi_n$ in \eqref{norms} are diagonal matrices with diagonal entries given by
\begin{equation}\label{Pis}
\Pi_{n,k}=\frac{\sigma_{n,k}^2}{\nu_{n,k}},\quad k=0,1,\ldots,n.
\end{equation}
The vector of polynomials
${\mathbb Q}_n=\left( Q_{n,0},  Q_{n,1},  \dots ,  Q_{n,n} \right)^{T}$ satisfy the three-term recurrence relations
\begin{equation}\label{TTRRq}
\begin{aligned}
-x\, {\mathbb Q}_n(x,y) & = A_{n,1} {\mathbb Q}_{n+1}(x,y)+ B_{n,1} {\mathbb Q}_n(x,y) + C_{n,1} {\mathbb Q}_{n-1}(x,y), \\
y\, {\mathbb Q}_n(x,y) & = A_{n,2} {\mathbb Q}_{n+1}(x,y)+ B_{n,2} {\mathbb Q}_n(x,y) + C_{n,2} {\mathbb Q}_{n-1}(x,y),
\end{aligned}
\end{equation}
where the elements in the coefficients $A_{n,1}, B_{n,1}, C_{n,1}$ (see \eqref{ABC1}) are given by
\begin{equation}\label{coeffs1}
\begin{aligned}
a_{n,k} & = \frac{(n-k+\alpha+1)(n+k+\alpha+\beta+\gamma+2)}{ (2n+\alpha+\beta+\gamma+2)_2},\quad k=0,1,\ldots,n, \\
b_{n,k} & = -(a_{n,k}+c_{n,k}),\quad k=0,1,\ldots,n,\\
c_{n,k} & =  \frac{ (n-k) (n+k+\beta+\gamma+1)}{(2n+\alpha+\beta+\gamma+1)_2},\quad k=0,1,\ldots,n-1,
\end{aligned}
\end{equation}
the elements in coefficient $A_{n,2}$ (see \eqref{ABC2}) are given by
\begin{equation}\label{coeffs2}
\begin{aligned}
a_{n,k}^{(1)} & = \frac{(n-k+\alpha+1)_2 \, k (k+\beta)}{(2n+\alpha+\beta+\gamma+2)_2 \, (2k+\beta+\gamma)_2},\quad k=1,\ldots,n, \\[5pt]
a_{n,k}^{(2)} & =\left(1+\frac{\beta^2-\gamma^2}{(2k+\beta+\gamma+2)(2k+\beta+\gamma)} \right) \frac{a_{n,k}}{2},\quad k=0,1,\ldots,n, \\[5pt]
a_{n,k}^{(3)} & = \frac{(n+k+\alpha+\beta+\gamma+2)_2\, (k+\gamma+1)(k+\beta+\gamma+1)}{(2n+\alpha+\beta+\gamma+2)_2\,  (2k+\beta+\gamma+1)_2},\quad k=0,1,\ldots,n,
\end{aligned}
\end{equation}
the ones in coefficient $B_{n,2}$ are
\begin{equation}\label{coeffs3}
\begin{aligned}
b_{n,k}^{(1)} & = \frac{2k (k+\beta)(n-k+\alpha+1)(n+k+\beta+\gamma+1)}{(2k+\beta+\gamma)_2\, (2n+\alpha+\beta+\gamma+1) (2n+\alpha+\beta+\gamma+3)},\quad k=1,\ldots,n,\\[5pt]
b_{n,k}^{(2)} & = \left(1+\frac{\beta^2-\gamma^2}{(2k+\beta+\gamma)(2k+\beta+\gamma+2)}\right) \frac{1+b_{n,k}}{2},\quad k=0,1,\ldots,n,\\[5pt]
b_{n,k}^{(3)} & = \frac{2(n-k)(n+k+\alpha+\beta+\gamma+2)(k+\gamma+1)(k+\beta+\gamma+1)}{ (2n+\alpha+\beta+\gamma+1) (2n+\alpha+\beta+\gamma+3) (2k+\beta+\gamma+1)_2},\quad k=0,1,\ldots,n-1,
\end{aligned}
\end{equation}
and in $C_{n,2}$
\begin{equation}\label{coeffs4}
\begin{aligned}
c_{n,k}^{(1)} & = \frac{(n+k+\beta+\gamma)_2\, (k+\beta) k}{(2n+\alpha+\beta+\gamma+1)_2\, (2k+\beta+\gamma)_2},\quad k=1,\ldots,n,\\[5pt]
c_{n,k}^{(2)} & = \left(1+\frac{\beta^2-\gamma^2}{(2k+\beta+\gamma+2)(2k+\beta+\gamma)} \right)\frac{c_{n,k}}{2},\quad k=0,1,\ldots,n-1,\\[5pt]
c_{n,k}^{(3)} & = \frac{(n-k-1)_2\,(k+\beta+\gamma+1)(k+\gamma+1)}{(2n+\alpha+\beta+\gamma+1)_2 (2k+\beta+\gamma+1)_2},\quad k=0,1,\ldots,n-2.
\end{aligned}
\end{equation}
It is possible to see that all entries of $A_{n,2},B_{n,2},C_{n,2}$ are nonnegative numbers. Evaluating the equations \eqref{TTRRq} at the point $(0,1)$ we get that the Jacobi matrix $J_{2}$ in \eqref{JacMat} is a \emph{stochastic matrix}. Therefore we get a nontrivial and non homogeneous discrete-time QBD process. In Figure 3 we can see a diagram of the possible transitions of this discrete-time QBD process.

%After relabeling the states $(n,k)$ where $n\geq0$ and $0\leq k\leq n$ in the following form
%$$
%(n,k)\to\binom{n+1}{2}+k,
%$$
%a diagram of the transitions between levels and phases is given by

%\vspace{0.6cm}
\begin{figure}
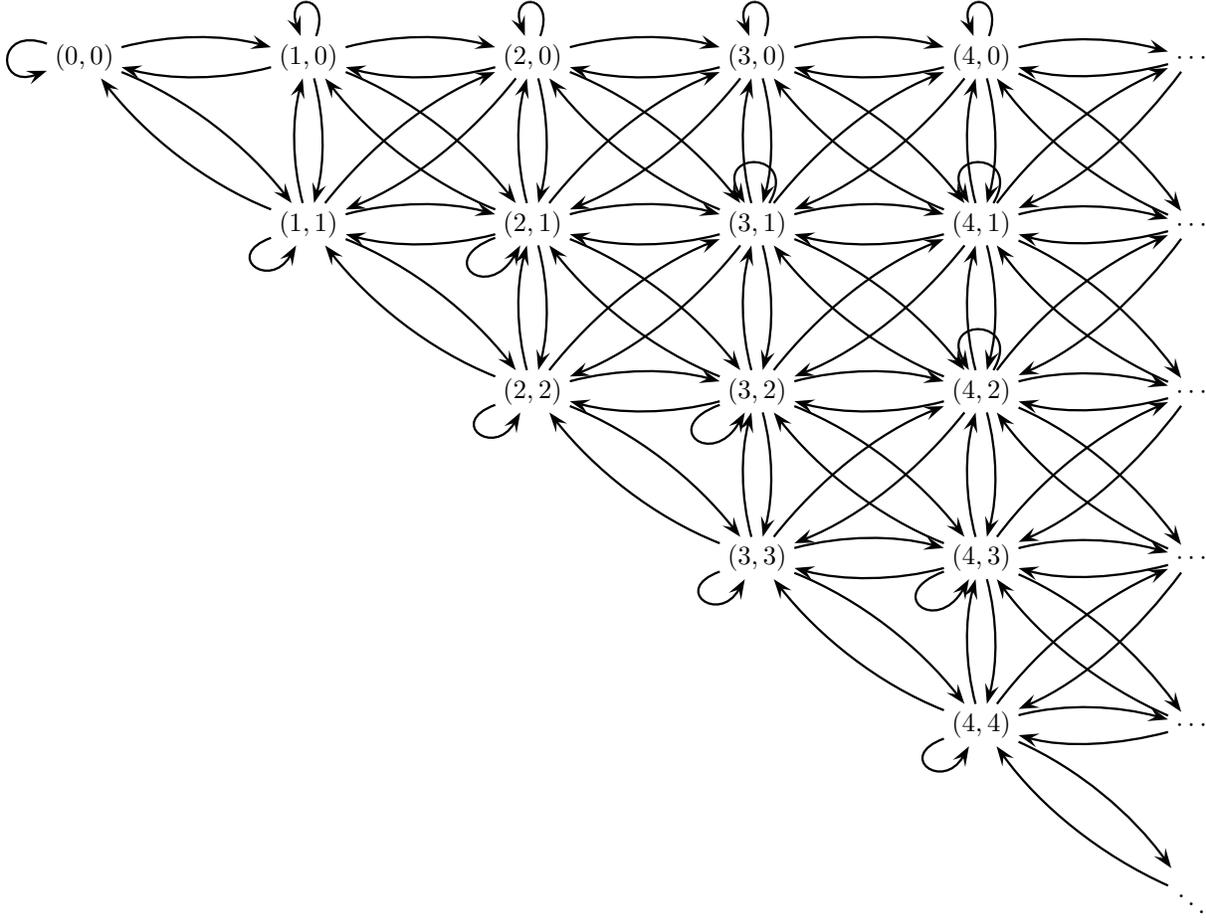

\begin{center}
$$\begin{psmatrix}[rowsep=1.8cm,colsep=2.2cm]
  \rnode{0}{(0,0)}& \rnode{1}{(1,0)} & \rnode{3}{(2,0)}& \rnode{6}{(3,0)}& \rnode{10}{(4,0)}& \rnode{15}{\Huge{\cdots}} \\
    & \rnode{2}{(1,1)} & \rnode{4}{(2,1)}& \rnode{7}{(3,1)}& \rnode{11}{(4,1)}& \rnode{16}{\Huge{\cdots}} \\
     & & \rnode{5}{(2,2)}& \rnode{8}{(3,2)}& \rnode{12}{(4,2)}& \rnode{17}{\Huge{\cdots}} \\
    & & & \rnode{9}{(3,3)}&\rnode{13}{(4,3)}& \rnode{18}{\Huge{\cdots}} \\
     & & & & \rnode{14}{(4,4)}& \rnode{19}{\Huge{\cdots}} \\
     & & & & & \rnode{20}{\ddots}
\psset{nodesep=3pt,arcangle=15,labelsep=2ex,linewidth=0.3mm,arrows=->,arrowsize=1mm
3} \nccurve[angleA=160,angleB=200,ncurv=4]{0}{0}
%\uput[u](-15.2,12.75){b_{0,0}^{(2)}}
\nccurve[angleA=70,angleB=110,ncurv=6]{1}{1}
\nccurve[angleA=70,angleB=110,ncurv=6]{3}{3}
\nccurve[angleA=70,angleB=110,ncurv=6]{6}{6}
\nccurve[angleA=70,angleB=110,ncurv=6]{10}{10}
\nccurve[angleA=200,angleB=240,ncurv=4]{2}{2}
\nccurve[angleA=200,angleB=240,ncurv=4]{5}{5}
\nccurve[angleA=200,angleB=240,ncurv=4]{9}{9}
\nccurve[angleA=200,angleB=240,ncurv=4]{14}{14}
\nccurve[angleA=200,angleB=240,ncurv=5]{4}{4}
\nccurve[angleA=200,angleB=240,ncurv=5]{8}{8}
\nccurve[angleA=200,angleB=240,ncurv=5]{13}{13}
\nccurve[angleA=60,angleB=120,ncurv=5]{7}{7}
\nccurve[angleA=60,angleB=120,ncurv=5]{11}{11}
\nccurve[angleA=60,angleB=120,ncurv=5]{12}{12}
 \ncarc{0}{1}\ncarc{1}{0} \ncarc{0}{2}\ncarc{2}{0} \ncarc{1}{2} \ncarc{2}{1}\ncarc{2}{3} \ncarc{3}{2} \ncarc{4}{1}\ncarc{1}{4}
 \ncarc{2}{5}\ncarc{5}{2}\ncarc{5}{9}\ncarc{9}{5}\ncarc{9}{14}\ncarc{14}{9}\ncarc{14}{20}\ncarc{20}{14}
\ncarc{3}{4} \ncarc{4}{3} \ncarc{4}{5}\ncarc{5}{4} \ncarc{7}{6} \ncarc{6}{7}
\ncarc{3}{1}\ncarc{1}{3}\ncarc{3}{6}\ncarc{6}{3}\ncarc{6}{10}\ncarc{10}{6}\ncarc{10}{15}\ncarc{15}{10}\ncarc{2}{4}\ncarc{4}{2}\ncarc{4}{7}\ncarc{7}{4}\ncarc{7}{11}\ncarc{11}{7}\ncarc{16}{11}\ncarc{11}{16}\ncarc{5}{8}\ncarc{8}{5}\ncarc{8}{12}\ncarc{12}{8}\ncarc{17}{12}\ncarc{12}{17}\ncarc{9}{13}\ncarc{13}{9}\ncarc{13}{18}\ncarc{18}{13}\ncarc{14}{19}\ncarc{19}{14}
\ncarc{7}{8}\ncarc{8}{7}\ncarc{9}{8}\ncarc{8}{9}\ncarc{10}{11}\ncarc{11}{10}\ncarc{11}{12}\ncarc{12}{11}\ncarc{13}{12}\ncarc{12}{13}\ncarc{13}{14}\ncarc{14}{13}
\ncarc{3}{7}\ncarc{7}{3}\ncarc{6}{4}\ncarc{4}{6}\ncarc{11}{6}\ncarc{6}{11}\ncarc{10}{7}\ncarc{7}{10}\ncarc{4}{8}\ncarc{8}{4}\ncarc{5}{7}\ncarc{7}{5}\ncarc{7}{12}\ncarc{12}{7}\ncarc{11}{8}\ncarc{8}{11}\ncarc{8}{13}\ncarc{13}{8}\ncarc{9}{12}\ncarc{12}{9}\ncarc{10}{16}\ncarc{16}{10}\ncarc{11}{15}\ncarc{15}{11}\ncarc{11}{17}\ncarc{17}{11}\ncarc{12}{16}\ncarc{16}{12}\ncarc{12}{18}\ncarc{18}{12}\ncarc{13}{17}\ncarc{17}{13}\ncarc{13}{19}\ncarc{19}{13}\ncarc{14}{18}\ncarc{18}{14}
%\psset{labelsep=-4.2ex}\nput{90}{0}{(0,0)}
\end{psmatrix}
$$
\end{center}
\caption{Diagram of all possible transitions of the discrete-time QBD process corresponding with $J_2$ for the orthogonal polynomials on the triangle.}
\end{figure}

The Jacobi matrix $J_{1}$ also have a probabilistic interpretation. Indeed, observe that the coefficients $a_{n,k}, c_{n,k}$ are always nonnegative (and bounded by 1) and $a_{n,k}+b_{n,k}+c_{n,k}=0$. That means that $J_{1}$ is the infinitesimal operator of a continuous-time QBD process. Since all coefficients are diagonal that means that transitions between phases are not possible. Therefore, for each phase $k$, the QBD process is a regular continuous-time birth-death process.

Now consider a Jacobi matrix of the form \eqref{Palpha}, i.e.
\begin{equation}\label{QBDsi}
\bm P=\tau_1 J_1+\tau_2 J_2.
\end{equation}
We want to give $\bm P$ a probabilistic interpretation. For that there are at least two possibilities, either a continuous or a discrete-time QBD process. If we want to have a continuous-time QBD process then we need $\bm P\bm e=\bm 0$ and nonnegative off-diagonal entries. But this is possible if and only if $\tau_2=0$ and $\tau_1>0$, i.e. a scalar multiple of $J_1$, which has all diagonal coefficients and the QBD process is trivial.

If we want to have a discrete-time QBD process then we need $\bm P\bm e=\bm e$ and nonnegative (scalar) entries. This is possible if and only if $\tau_2=1$ and the parameter $\tau_1$ is chosen in such a way that all entries of $\bm P$ are nonnegative. For simplicity, we will call $\tau=\tau_1$. Bearing in mind the shape of the coefficients $A_{n,i}, B_{n,i}, C_{n,i}, n\geq0, i=1,2,$ in \eqref{ABC1} and \eqref{ABC2} and looking at their entries in \eqref{coeffs1}--\eqref{coeffs4}, the entries of $\bm P=\tau J_1+J_2$ are nonnegative if and only if
\begin{equation*}
\begin{split}
\tau a_{n,k}+a_{n,k}^{(2)}&\geq0,\\
\tau(a_{n,k}+c_{n,k})&\leq b_{n,k}^{(2)},\\
\tau c_{n,k}+c_{n,k}^{(2)}&\geq0,
\end{split}\quad\mbox{for all}\quad n\geq0,\quad k=0,1,\ldots,n.
\end{equation*}
In other words,
\begin{equation*}
\begin{split}
\tau &\geq -D_k^{\beta,\gamma},\\
\tau&\leq D_k^{\beta,\gamma}\left(-1+\frac{1}{a_{n,k}+c_{n,k}}\right),
\end{split}\quad\mbox{for all}\quad n\geq0,\quad k=0,1,\ldots,n,
\end{equation*}
where
\begin{equation}\label{ddk}
D_k^{\beta,\gamma}=\frac{1}{2}\left(1+\frac{\beta^2-\gamma^2}{(2k+\beta+\gamma+2)(2k+\beta+\gamma)} \right).
\end{equation}
From \eqref{coeffs1}, we observe that 
$$
a_{n,k}+c_{n,k}=\frac{1}{2}\left(1+\frac{\alpha^2-(2k+\beta+\gamma+1)^2} {(2n+\alpha+\beta+\gamma+1)(2n+\alpha+\beta+\gamma+3)}\right).
$$
On one hand, we have 
$$
\min_{0\leq k\leq n}\left\{D_k^{\beta,\gamma}\right\}=\begin{cases}
1/2,&\mbox{if}\quad  \beta^2\geq\gamma^2,\\
\\
\displaystyle\frac{\beta+1}{\beta+\gamma+2},&\mbox{if} \quad \beta^2<\gamma^2.
\end{cases}
$$
On the other hand, it is possible to see that
$$
\min_{n\in\mathbb{N}_0,0\leq k\leq n}\left\{-1+\frac{1}{a_{n,k}+c_{n,k}}\right\}=\begin{cases}
K_{\alpha,\beta,\gamma},&\mbox{if}\quad  \alpha<-(\beta+\gamma+1),\\
\\
1,&\mbox{if}\quad  \alpha^2\leq(\beta+\gamma+1)^2,\\
\\
\displaystyle\frac{2+\beta+\gamma}{\alpha+1},&\mbox{if} \quad \alpha>\beta+\gamma+1,
\end{cases}
$$
where
$$
K_{\alpha,\beta,\gamma}=1-\frac{\alpha^2-(\beta+\gamma+1)^2}{4+(\alpha+3)(\alpha+\beta+\gamma+1)}.
$$
Combining these two relations we have that the entries of $\bm P=\tau J_1+J_2$ are nonnegative (and therefore $\bm P$ is a stochastic matrix) if and only if the upper bound of $\tau$ is given by
\begin{equation}\label{condaa21}
\tau\leq
\begin{cases}
K_{\alpha,\beta,\gamma}/2,&\mbox{if}\quad  \beta^2\geq\gamma^2\;\;\mbox{and}\;\; \alpha<-(\beta+\gamma+1),\\
\\1/2,&\mbox{if}\quad  \beta^2\geq\gamma^2\;\;\mbox{and}\;\; \alpha^2\leq(\beta+\gamma+1)^2,\\
\\
\displaystyle\frac{\beta+\gamma+2}{2(\alpha+1)},&\mbox{if} \quad\beta^2\geq\gamma^2\;\;\mbox{and}\;\; \alpha>\beta+\gamma+1,\\
\\
\displaystyle\frac{(\beta+1)K_{\alpha,\beta,\gamma}}{\beta+\gamma+2},&\mbox{if} \quad \beta^2<\gamma^2\;\;\mbox{and}\;\;\alpha<-(\beta+\gamma+1),\\
\\
\displaystyle\frac{\beta+1}{\beta+\gamma+2},&\mbox{if} \quad \beta^2<\gamma^2\;\;\mbox{and}\;\;\alpha^2\leq(\beta+\gamma+1)^2,\\
\\
\displaystyle\frac{\beta+1}{\alpha+1},&\mbox{if} \quad \beta^2<\gamma^2\;\;\mbox{and}\;\; \alpha>\beta+\gamma+1,
\end{cases}
\end{equation}
while the lower bound of $\tau$ is given by
\begin{equation}\label{condaa2}
\tau\geq
\begin{cases}
-1/2,&\mbox{if}\quad  \beta^2\geq\gamma^2,\\
\\
-\displaystyle\frac{\beta+1}{\beta+\gamma+2},&\mbox{if} \quad \beta^2<\gamma^2.
\end{cases}
\end{equation}
Therefore, for all values of $\tau$ in the range \eqref{condaa21} and \eqref{condaa2}, we have a \emph{family} of discrete-time QBD process with transition probability matrix $\bm P=\tau J_1+J_2$. Thus the Karlin-McGregor representation formula \eqref{KMcF} for the $(i,j)$ block entry of the matrix $\bm P$ is given by
\begin{equation*}
\bm P_{i,j}^n=\frac{\Gamma(\alpha+\beta+\gamma+3)}{\Gamma(\alpha+1) \Gamma(\beta+1)\Gamma(\gamma+1)}\left(\int_\mathbb{T}(y-\tau x)^n\mathbb{Q}_i(x,y)\mathbb{Q}_j^T(x,y)x^{\alpha} y^{\beta} (1-x-y)^{\gamma}dxdy\right)\Pi_j,
\end{equation*}
where $\Pi_j$ is a diagonal matrix with entries given by \eqref{Pis}. From \eqref{KMcFd} and \eqref{poltri} we can derive a separated expression for all probabilities, given by
\begin{align*}
\left(\bm P_{i,j}^n\right)_{i',j'}=&\frac{\Gamma(\alpha+\beta+\gamma+3)}{\Gamma(\alpha+1) \Gamma(\beta+1)\Gamma(\gamma+1)}\frac{\Pi_{j,j'}}{\sigma_{i,i'}\sigma_{j,j'}}\sum_{k=0}^n\binom{n}{k}(-1)^k\tau^k\\
&\times\left(\int_\mathbb{T} x^{\alpha+k}y^{\beta+n-k}P_{i-i'}^{(2i'+\beta+\gamma+1,\alpha)}(2x-1)P_{j-j'}^{(2j'+\beta+\gamma+1,\alpha)}(2x-1)\right.\\
&\hspace{1cm}\times \left.P_{i'}^{(\gamma,\beta)}\left(\frac{2y}{\sqrt{1-x}}-1\right)P_{j'}^{(\gamma,\beta)}\left(\frac{2y}{\sqrt{1-x}}-1\right)(1-x)^{i'+j'}(1-x-y)^\gamma dxdy\right).
\end{align*}
According to Theorem \ref{Teo} we can construct an invariant measure $\bm\pi$ for the QBD process given by \eqref{ID}. Finally, the family of discrete-time QBD processes is recurrent (see \eqref{recu}) if and only if
$$
\int_\mathbb{T}\frac{x^{\alpha} y^{\beta} (1-x-y)^{\gamma}}{1-y+\tau x}dxdy=\infty.
$$
After some computations it turns out that, in the range of the values of $\tau$ in \eqref{condaa21} and \eqref{condaa2}, this integral is divergent if and only if $\alpha+\gamma\leq-1$. Otherwise the QBD process is transient. The QBD process can never be positive recurrent since the spectral matrix is absolutely continuous and does not have any jumps.

\subsection{Normalization at the point $(0,0)$}\label{secnorm0}

In this case, using again \eqref{jac-norm}, the coefficients $\sigma_{n,k}$ are given by
$$
\sigma_{n,k}=P_{n,k}(0,0)= P_{n-k}^{(2k+\beta+\gamma+1,\alpha)}(-1) \, P_k^{(\gamma,\beta)}(-1)= (-1)^{n}  \frac{(\alpha+1)_{n-k}}{(n-k)!} \frac{(\beta+1)_k}{k!}.
$$
Therefore, the polynomials $Q_{n,k}(x,y)=\sigma_{n,k}^{-1}P_{n,k}(x,y)$ satisfy $Q_{n,k}(0,0)=1$ for all $n\geq0$ and $0\leq k\leq n$. The inverse of the square norms can be computed as in \eqref{Pis}. The vector of polynomials
${\mathbb Q}_n=\left( Q_{n,0},  Q_{n,1},  \dots ,  Q_{n,n} \right)^{T}$ satisfies now the three-term recurrence relations
\begin{equation}\label{TTRRq2}
\begin{aligned}
-x\, {\mathbb Q}_n(x,y) & = A_{n,1} {\mathbb Q}_{n+1}(x,y)+ B_{n,1} {\mathbb Q}_n(x,y) + C_{n,1} {\mathbb Q}_{n-1}(x,y), \\
-y\, {\mathbb Q}_n(x,y) & = A_{n,2} {\mathbb Q}_{n+1}(x,y)+ B_{n,2} {\mathbb Q}_n(x,y) + C_{n,2} {\mathbb Q}_{n-1}(x,y).
\end{aligned}
\end{equation}
The coefficients of \eqref{ABC1} and \eqref{ABC2} are exactly the same as in the previous case, i.e. \eqref{coeffs1}--\eqref{coeffs4} interchanging $\beta$ by $\gamma$, except for the coefficients $a_{n,k}^{(2)}, b_{n,k}^{(2)}$ and $c_{n,k}^{(2)}$, where it appears a minus sign (but not interchanging $\beta$ by $\gamma$).

In this case we have, evaluating at $(0,0)$ in \eqref{TTRRq2}, that $J_1\bm e=J_2\bm e=\bm 0$. $J_1$ is the same matrix as before, so it represents a trivial continuous-time QBD process. Nevertheless, although $J_2\bm e=\bm 0$, $J_2$ \emph{does not generate} a continuous-time QBD process itself since $a_{n,k}^{(2)}\leq0$ and $c_{n,k}^{(2)}\leq0$.

Consider now the Jacobi matrix
$$
\bm{\mathcal{A}}=\tau_1J_1+\tau_2J_2.
$$
In order to have a continuous-time QBD process (now it can not be a discrete-time QBD process) we need that $\bm{\mathcal{A}}\bm e=\bm 0$ (which is always satisfied) and \emph{all} nonnegative off-diagonal entries. This holds if and only if $\tau_2\geq0$ and
\begin{equation*}
\begin{split}
\tau_1 a_{n,k}+\tau_2a_{n,k}^{(2)}&\geq0,\\
\tau_1 c_{n,k}+\tau_2c_{n,k}^{(2)}&\geq0,
\end{split}\quad\mbox{for all}\quad n\geq0,\quad k=0,1,\ldots,n.
\end{equation*}
This is equivalent to
$$
\tau_1\geq \tau_2\max_{0\leq k\leq n}\left\{D_k^{\beta,\gamma}\right\},
$$
where, $D_k^{\beta,\gamma}$ is defined by \eqref{ddk}. In other words
\begin{equation}\label{condaa22}
\frac{\tau_1}{\tau_2}\geq
\begin{cases}
1/2,&\mbox{if}\quad  \beta^2\leq\gamma^2,\\
\\
\displaystyle\frac{\beta+1}{\beta+\gamma+2},&\mbox{if} \quad \beta^2>\gamma^2.
\end{cases}
\end{equation}
If $\tau_2=0$ then we need $\tau_1\geq0$. A diagram of the possible transition for this continuous-time QBD process is the same in Figure 3, but without self transitions.

Therefore, for all values of $\tau_1$ and $\tau_2$ in the range \eqref{condaa22}, we have again a \emph{family} of continuous-time QBD process with infinitesimal operator matrix $\bm{\mathcal{A}}=\tau_1 J_1+\tau_2J_2$. Thus, the Karlin-McGregor representation formula \eqref{KMcF} for the $(i,j)$ block entry of the transition function matrix $\bm P(t)$ is given by
\begin{equation*}
\bm P_{i,j}(t)=\frac{\Gamma(\alpha+\beta+\gamma+3)}{\Gamma(\alpha+1) \Gamma(\beta+1)\Gamma(\gamma+1)}\left(\int_\mathbb{T}e^{-(\tau_1x+\tau_2y)t}\mathbb{Q}_i(x,y)\mathbb{Q}_j^T(x,y)x^{\alpha} y^{\beta} (1-x-y)^{\gamma}dxdy\right)\Pi_j,
\end{equation*}
where $\Pi_j$ is a diagonal matrix with entries given by \eqref{Pis}. From \eqref{KMcFd2} and \eqref{poltri} we can derive a separated expression for all probabilities, given by
\begin{align*}
\left(\bm P_{i,j}(t)\right)_{i',j'}=&\frac{\Gamma(\alpha+\beta+\gamma+3)}{\Gamma(\alpha+1) \Gamma(\beta+1)\Gamma(\gamma+1)}\frac{\Pi_{j,j'}}{\sigma_{i,i'}\sigma_{j,j'}}\\
&\times\left(\int_\mathbb{T} e^{-(\tau_1x+\tau_2y)t}x^{\alpha}y^{\beta}P_{i-i'}^{(2i'+\beta+\gamma+1,\alpha)}(2x-1)P_{j-j'}^{(2j'+\beta+\gamma+1,\alpha)}(2x-1)\right.\\
&\hspace{1cm}\times \left.P_{i'}^{(\gamma,\beta)}\left(\frac{2y}{\sqrt{1-x}}-1\right)P_{j'}^{(\gamma,\beta)}\left(\frac{2y}{\sqrt{1-x}}-1\right)(1-x)^{i'+j'}(1-x-y)^\gamma dxdy\right).
\end{align*}
According to Theorem \ref{Teo} we can construct an invariant measure $\bm\pi$ for the QBD process given by \eqref{ID}. Finally, the family of continuous-time QBD processes is recurrent (see \eqref{recu2}) if and only if
$$
\int_\mathbb{T}\frac{x^{\alpha} y^{\beta} (1-x-y)^{\gamma}}{\tau_1x+\tau_2y}dxdy=\infty.
$$
After some computations it turns out that, in the range of the values of $\tau_1$ in \eqref{condaa22}, this integral is divergent if and only if $\alpha+\beta\leq-1$. If $\tau_2=0$ and $\tau_1>0$ the divergence is equivalent to $\alpha\leq0$. Otherwise the QBD process is transient. The QBD process can never be positive recurrent since the spectral measure is absolutely continuous and does not have any jumps.

\subsection{An urn model for the orthogonal polynomials on the triangle}
In this section we will give a probabilistic interpretation of one of the QBD models introduced in Section \ref{secnorm1}. For simplicity, we will study the case of the discrete-time QBD process \eqref{QBDsi} with $\tau_1=0$ and $\tau_2=1$ (therefore $\bm P=J_2$).

As we can see from \eqref{coeffs2}--\eqref{coeffs4}, the probability coefficients are quite complicated and depend on three parameters $\alpha, \beta, \gamma$, apart from the level $n$ and phase $k$. However, we managed to find an urn model for this QBD process by decomposing it into two simpler urn models. For that, we will try to get a stochastic block LU factorization of the Jacobi matrix $J_2$. The spirit of this method is the same as the one used in \cite{GdI1,GdI2}. Write $J_2$ in \eqref{JacMat} as
\begin{equation}\label{JLU}
J_2=\left(\begin{array}{cccccc}
S_0   &    &         &       \bigcirc \\
R_1   &S_1   &                   \\
          &R_2   & S_2 &          \\
  \bigcirc         &           & \ddots  & \ddots
\end{array}\right)\left(\begin{array}{cccccc}
Y_0   & X_0  &         &       &\bigcirc \\
   &Y_1  & X_1 &       &          \\
          &  & Y_2 &X_2 &         \\
  \bigcirc         &           & & \ddots & \ddots
\end{array}\right)=J_LJ_U.
\end{equation}
A direct computation shows that
\begin{equation}\label{FLU}
\begin{split}
A_{n,2}&=S_{n}X_n,\quad n\geq0,\\
B_{n,2}&=R_{n}X_{n-1}+S_nY_n,\quad n\geq0,\\
C_{n,2}&=R_nY_{n-1},\quad n\geq1.
\end{split}
\end{equation}
Since $A_{n,2}, B_{n,2},C_{n,2}$ are matrices of dimension $(n+1)\times(n+2)$, $(n+1)\times(n+1)$ and $(n+1)\times n$, respectively, we have that $X_n, Y_n$ are matrices of dimension $(n+1)\times(n+2)$ and $(n+1)\times(n+1)$, respectively, and $S_n, R_n$ are matrices of dimension $(n+1)\times(n+1)$ and $(n+1)\times n$, respectively. We found that \emph{one} solution of equations \eqref{FLU} is given by coefficients $X_n, Y_n, S_n, R_n$, where
\begin{equation*}
X_n=\left[
\begin{array}{ccccccc}
x_{n,0}^{(2)} & x_{n,0}^{(3)}&         &          &  \\
& x_{n,1}^{(2)}& x_{n,1}^{(3)} &         &    \\
        && \ddots  & \ddots  &         \\
 &        && x_{n,n}^{(2)} & x_{n,n}^{(3)}
\end{array}
\right],
\quad
Y_n=\left[
\begin{array}{ccccccc}
y_{n,0}^{(2)} & y_{n,0}^{(3)}&         &        &  \\
 & y_{n,1}^{(2)}& y_{n,1}^{(3)} &        &           \\
              &  & \ddots  & \ddots &           \\
              &        & &  y_{n,n-1}^{(2)} & y_{n,n-1}^{(3)}  \\
              &        &       &       & y_{n,n}^{(2)}
\end{array}
\right],
\end{equation*}
and
\begin{equation*}
S_n=\left[
\begin{array}{ccccccc}
s_{n,0}^{(2)} & &         &        &  \\
s_{n,1}^{(1)} & s_{n,1}^{(2)}&&        &           \\
              & \ddots & \ddots  &  &           \\
              &        & s_{n,n-1}^{(1)} &  s_{n,n-1}^{(2)} &   \\
              &        &       &  s_{n,n}^{(1)}     & s_{n,n}^{(2)}
\end{array}
\right],\quad
R_n=\left[
\begin{array}{ccccccc}
r_{n,0}^{(2)} & &         &          &  \\
r_{n,1}^{(1)} & r_{n,1}^{(2)}& &         &    \\
        & \ddots & \ddots  &  &         \\
         &        & r_{n,n-2}^{(1)} & r_{n,n-2}^{(2)} &  \\
         &        &        &  r_{n,n-1}^{(1)} & r_{n,n-1}^{(2)} \\
        &        &        &         & r_{n,n}^{(1)}
\end{array}
\right].
\end{equation*}
The elements in $X_n, Y_n$ are given by
\begin{equation}\label{xyprob}
\begin{split}
x_{n,k}^{(2)}&=\frac{(n-k+\alpha+1)(\beta+k+1)}{(2n+\alpha+\beta+\gamma+3)(\beta+\gamma+2k+2)},\quad k=0,1,\ldots,n,\\
x_{n,k}^{(3)}&=\frac{(n+k+\alpha+\beta+\gamma+3)(\gamma+k+1)}{(2n+\alpha+\beta+\gamma+3)(\beta+\gamma+2k+2)},\quad k=0,1,\ldots,n,\\
y_{n,k}^{(2)}&=\frac{(n+k+\beta+\gamma+2)(\beta+k+1)}{(2n+\alpha+\beta+\gamma+3)(\beta+\gamma+2k+2)},\quad k=0,1,\ldots,n,\\
y_{n,k}^{(3)}&=\frac{(n-k)(\gamma+k+1)}{(2n+\alpha+\beta+\gamma+3)(\beta+\gamma+2k+2)},\quad k=0,1,\ldots,n-1,
\end{split}
\end{equation}
while the elements of $S_n, R_n$ are given by
\begin{equation}\label{srprob}
\begin{split}
s_{n,k}^{(1)}&=\frac{k(n-k+\alpha+1)}{(2n+\alpha+\beta+\gamma+2)(\beta+\gamma+2k+1)},\quad k=1,\ldots,n,\\
s_{n,k}^{(2)}&=\frac{(n+k+\alpha+\beta+\gamma+2)(\beta+\gamma+k+1)}{(2n+\alpha+\beta+\gamma+2)(\beta+\gamma+2k+1)},\quad k=0,1,\ldots,n,\\
r_{n,k}^{(1)}&=\frac{k(n+k+\beta+\gamma+1)}{(2n+\alpha+\beta+\gamma+2)(\beta+\gamma+2k+1)},\quad k=1,\ldots,n,\\
r_{n,k}^{(2)}&=\frac{(n-k)(\beta+\gamma+k+1)}{(2n+\alpha+\beta+\gamma+2)(\beta+\gamma+2k+1)},\quad k=0,1,\ldots,n-1.
\end{split}
\end{equation}
Observe the important simplification of these elements compared with \eqref{coeffs2}--\eqref{coeffs4}. Another important observation is that $J_L$ and $J_U$ are \emph{also stochastic matrices}, so each one of them is again a discrete-time QBD process.

\begin{remark}
The stochastic LU factorization in \eqref{JLU} is not necessarily unique, but it is certainly one that simplifies all computations significantly. Similar considerations apply if we take into account a stochastic UL factorization. It is possible to see that the elements of the factors $X_n, Y_n, S_n, R_n$ for the UL factorization (at least one) are the same as the ones of the LU factorization but replacing $\beta$ by $\beta-1$. For more information about stochastic UL or LU factorizations see \cite{GdI1,GdI2}.
\end{remark}

From now on, we will assume that $\alpha,\beta,\gamma$ are nonnegative integers. Consider $\{Z_t : t=0,1,\ldots\}$ the discrete-time QBD process on the state space $\{(n,k) : 0\leq k\leq n, n\in\mathbb{N}_0\}$ whose one-step transition probability matrix is given by the coefficients $A_{n,2},B_{n,2},C_{n,2}$ in \eqref{coeffs2}--\eqref{coeffs4} (see also \eqref{JacMat} and \eqref{ABC2}). Consider the LU block factorization \eqref{JLU} $J_{2}=J_LJ_U$. Each of these matrices $J_L$ and $J_U$ will give rise to an urn experiment which we call Experiment 1 and Experiment 2, respectively. At every time step $t=0,1,2,\ldots$ the state $(n,k)$ will represent the number of $n$ blue balls inside the $k$-th urn $\mbox{A}_k, k=0,1,\ldots,n$. Observe that the number of urns available goes with the number of blue balls at every time step. All the urns we use in both experiments sit in a bath consisting of an infinite number of blue and red balls.

Experiment 1 (for $J_L$) will give rise to a discrete-time pure-death QBD process $\{Z_t^{(1)} : t=0,1,\ldots\}$ on $\{(n,k) : 0\leq k\leq n, n\in\mathbb{N}_0\}$ with diagram given by Figure 4.
%\vspace{1cm}
\begin{figure}
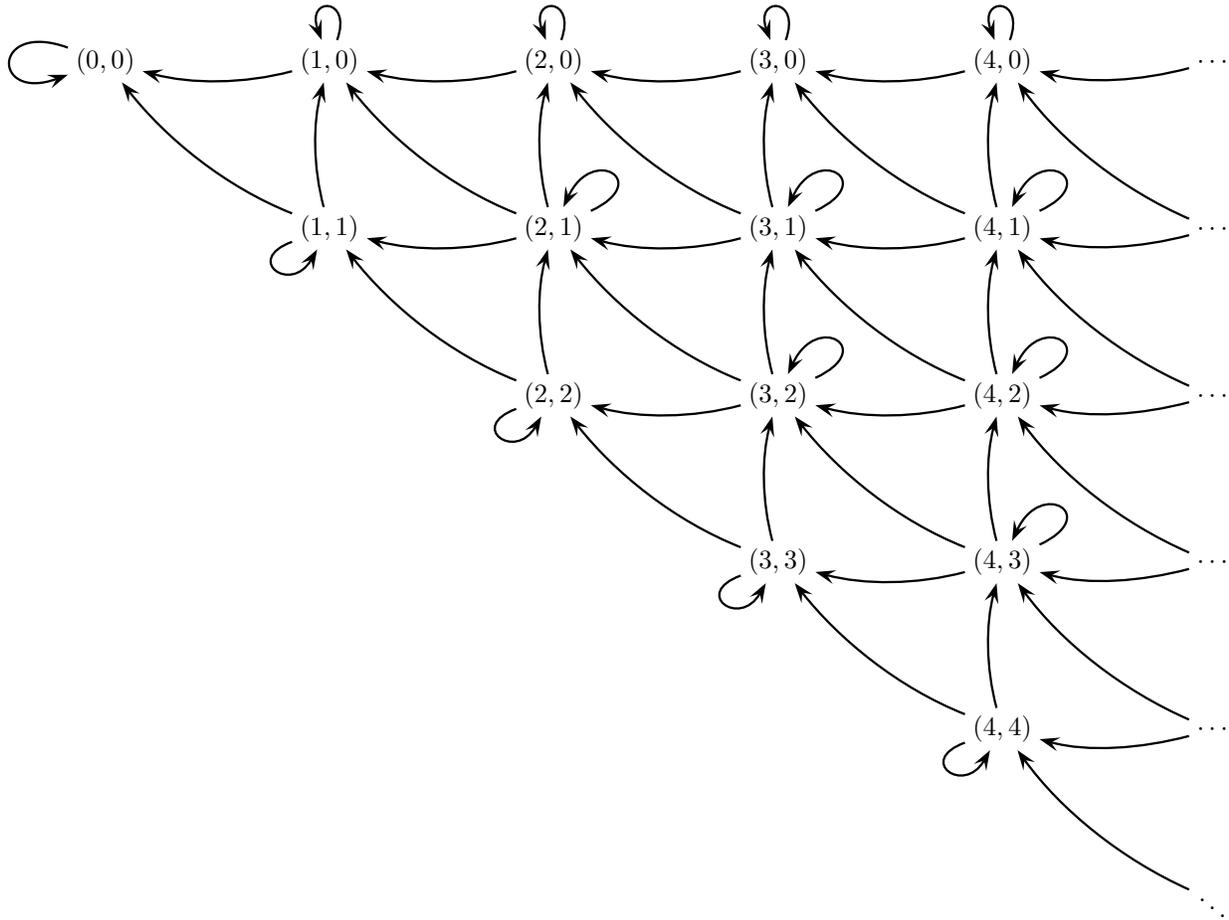

\begin{center}
$$\begin{psmatrix}[rowsep=1.8cm,colsep=2.2cm]
  \rnode{0}{(0,0)}& \rnode{1}{(1,0)} & \rnode{3}{(2,0)}& \rnode{6}{(3,0)}& \rnode{10}{(4,0)}& \rnode{15}{\Huge{\cdots}} \\
    & \rnode{2}{(1,1)} & \rnode{4}{(2,1)}& \rnode{7}{(3,1)}& \rnode{11}{(4,1)}& \rnode{16}{\Huge{\cdots}} \\
     & & \rnode{5}{(2,2)}& \rnode{8}{(3,2)}& \rnode{12}{(4,2)}& \rnode{17}{\Huge{\cdots}} \\
    & & & \rnode{9}{(3,3)}&\rnode{13}{(4,3)}& \rnode{18}{\Huge{\cdots}} \\
     & & & & \rnode{14}{(4,4)}& \rnode{19}{\Huge{\cdots}} \\
     & & & & & \rnode{20}{\ddots}
\psset{nodesep=3pt,arcangle=15,labelsep=2ex,linewidth=0.3mm,arrows=->,arrowsize=1mm
3} \nccurve[angleA=160,angleB=200,ncurv=6]{0}{0}
\nccurve[angleA=70,angleB=110,ncurv=6]{1}{1}
\nccurve[angleA=70,angleB=110,ncurv=6]{3}{3}
\nccurve[angleA=70,angleB=110,ncurv=6]{6}{6}
\nccurve[angleA=70,angleB=110,ncurv=6]{10}{10}
\nccurve[angleA=200,angleB=240,ncurv=4]{2}{2}
\nccurve[angleA=200,angleB=240,ncurv=4]{5}{5}
\nccurve[angleA=200,angleB=240,ncurv=4]{9}{9}
\nccurve[angleA=200,angleB=240,ncurv=4]{14}{14}
\nccurve[angleA=25,angleB=65,ncurv=5]{4}{4}
\nccurve[angleA=25,angleB=65,ncurv=5]{8}{8}
\nccurve[angleA=25,angleB=65,ncurv=5]{13}{13}
\nccurve[angleA=25,angleB=65,ncurv=5]{7}{7}
\nccurve[angleA=25,angleB=65,ncurv=5]{11}{11}
\nccurve[angleA=25,angleB=65,ncurv=5]{12}{12}
\ncarc{1}{0}\ncarc{2}{0} \ncarc{2}{1}\ncarc{4}{1}
\ncarc{5}{2}\ncarc{9}{5}\ncarc{14}{9}\ncarc{20}{14}
\ncarc{4}{3}\ncarc{5}{4} \ncarc{7}{6}
\ncarc{3}{1}\ncarc{6}{3}\ncarc{10}{6}\ncarc{15}{10}\ncarc{4}{2}\ncarc{7}{4}\ncarc{11}{7}\ncarc{16}{11}\ncarc{8}{5}\ncarc{12}{8}\ncarc{17}{12}\ncarc{13}{9}\ncarc{18}{13}\ncarc{19}{14}
\ncarc{8}{7}\ncarc{9}{8}\ncarc{11}{10}\ncarc{12}{11}\ncarc{13}{12}\ncarc{14}{13}
\ncarc{7}{3}\ncarc{11}{6}\ncarc{8}{4}\ncarc{12}{7}\ncarc{13}{8}\ncarc{16}{10}\ncarc{17}{11}\ncarc{18}{12}\ncarc{19}{13}
%\psset{labelsep=-4.2ex}\nput{90}{0}{(0,0)}
\end{psmatrix}
$$
\end{center}
\caption{Diagram of all possible transitions of the discrete-time pure-death QBD process generated by the Jacobi matrix $J_L$.}
\end{figure}
The initial state is $(n,k)$, where $n$ is the number of blue balls inside the $k$-th urn $\mbox{A}_k$. Remove all the balls and put $k$ blue balls and $k+\beta+\gamma+1$ red balls in the urn $\mbox{A}_k$. Draw one ball from the urn at random with the uniform distribution. We have two possibilities:
\begin{enumerate}
\item If we get a blue ball then we remove/add balls until we have $n-k+\alpha+1$ blue balls and $n+k+\beta+\gamma+1$ red balls in the urn $\mbox{A}_k$. Then we draw again one ball from the urn at random with the uniform distribution and we have two possibilities:
\begin{itemize}
\item If we get a blue ball then we remove all balls in urn $\mbox{A}_{k}$ and add $n$ blue balls to the urn $\mbox{A}_{k-1}$ and start over. Therefore, joining both steps, we have
$$
\mathbb{P}\left[Z_1^{(1)}=(n,k-1)\; |\; Z_0^{(1)}=(n,k)\right]=\frac{k}{\beta+\gamma+2k+1}\frac{n-k+\alpha+1}{2n+\alpha+\beta+\gamma+2}.
$$
Observe that this probability is given by $s_{n,k}^{(1)}$ in \eqref{srprob}.
\item If we get a red ball then we remove all balls in urn $\mbox{A}_{k}$ and add $n-1$ blue balls to the urn $\mbox{A}_{k-1}$ and start over. Therefore, joining both steps, we have
$$
\mathbb{P}\left[Z_1^{(1)}=(n-1,k-1)\; |\; Z_0^{(1)}=(n,k)\right]=\frac{k}{\beta+\gamma+2k+1}\frac{n+k+\beta+\gamma+1}{2n+\alpha+\beta+\gamma+2}.
$$
Observe that this probability is given by $r_{n,k}^{(1)}$ in \eqref{srprob}.
\end{itemize}
\item If we get a red ball then we remove/add balls until we have $n+k+\alpha+\beta+\gamma+2$ blue balls and $n-k$ red balls in the urn $\mbox{A}_k$. Then we draw again one ball from the urn at random with the uniform distribution and we have two possibilities:
\begin{itemize}
\item If we get a blue ball then we remove/add balls in urn $\mbox{A}_{k}$ until we have $n$ blue balls in urn $\mbox{A}_{k}$ and start over. Therefore, joining both steps, we have
$$
\mathbb{P}\left[Z_1^{(1)}=(n,k)\; |\; Z_0^{(1)}=(n,k)\right]=\frac{\beta+\gamma+k+1}{\beta+\gamma+2k+1}\frac{n+k+\alpha+\beta+\gamma+2}{2n+\alpha+\beta+\gamma+2}.
$$
Observe that this probability is given by $s_{n,k}^{(2)}$ in \eqref{srprob}.
\item If we get a red ball then we remove/add balls in urn $\mbox{A}_{k}$ until we have $n-1$ blue balls in urn $\mbox{A}_{k}$ and start over. Therefore, joining both steps, we have
$$
\mathbb{P}\left[Z_1^{(1)}=(n-1,k)\; |\; Z_0^{(1)}=(n,k)\right]=\frac{\beta+\gamma+k+1}{\beta+\gamma+2k+1}\frac{n-k}{2n+\alpha+\beta+\gamma+2}.
$$
Observe that this probability is given by $r_{n,k}^{(2)}$ in \eqref{srprob}.
\end{itemize}
\end{enumerate}

Experiment 2 (for $J_U$) is similar but it will give rise to a discrete-time pure-birth QBD process $\{Z_t^{(2)} : t=0,1,\ldots\}$ on $\{(n,k) : 0\leq k\leq n, n\in\mathbb{N}_0\}$ with diagram given by Figure 5.
\begin{figure}
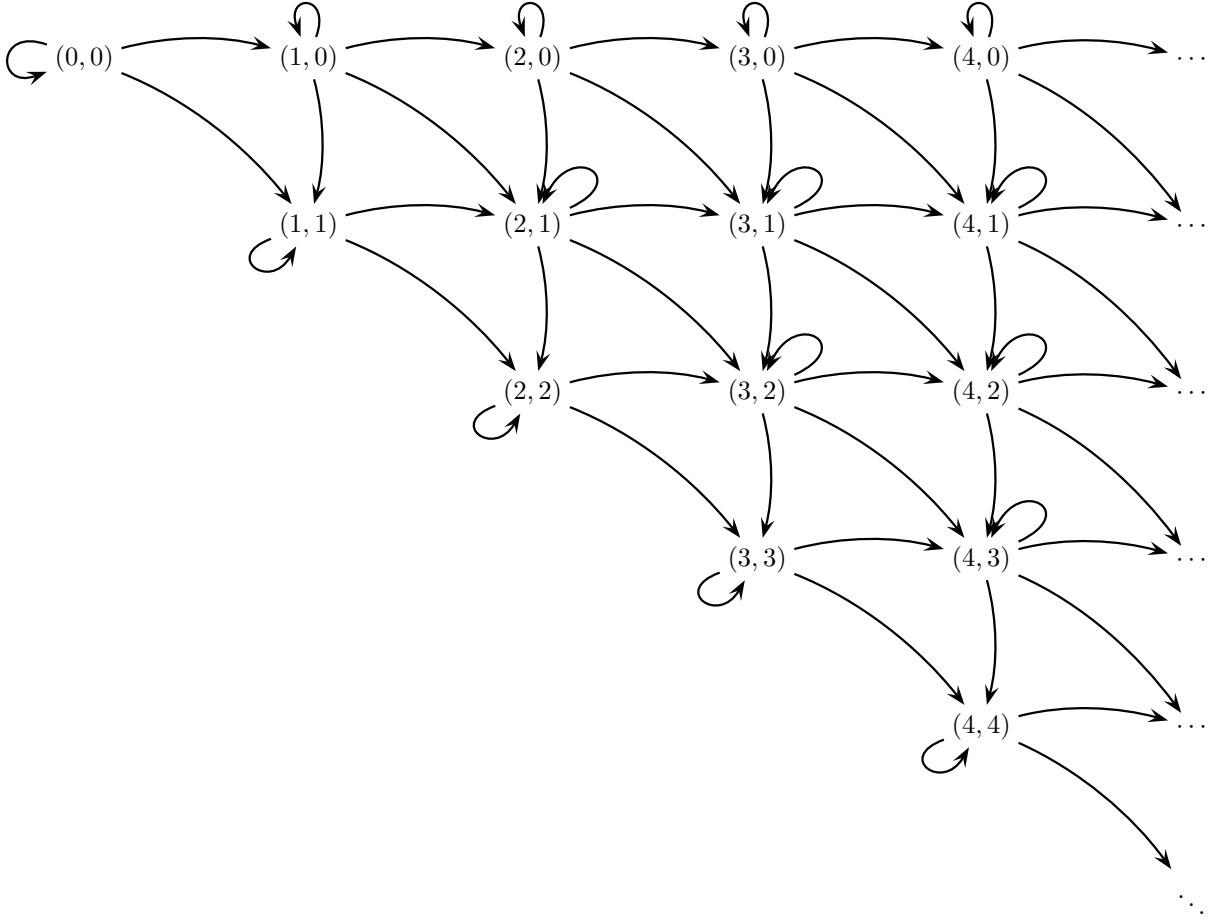

\begin{center}
$$\begin{psmatrix}[rowsep=1.8cm,colsep=2.2cm]
  \rnode{0}{(0,0)}& \rnode{1}{(1,0)} & \rnode{3}{(2,0)}& \rnode{6}{(3,0)}& \rnode{10}{(4,0)}& \rnode{15}{\Huge{\cdots}} \\
    & \rnode{2}{(1,1)} & \rnode{4}{(2,1)}& \rnode{7}{(3,1)}& \rnode{11}{(4,1)}& \rnode{16}{\Huge{\cdots}} \\
     & & \rnode{5}{(2,2)}& \rnode{8}{(3,2)}& \rnode{12}{(4,2)}& \rnode{17}{\Huge{\cdots}} \\
    & & & \rnode{9}{(3,3)}&\rnode{13}{(4,3)}& \rnode{18}{\Huge{\cdots}} \\
     & & & & \rnode{14}{(4,4)}& \rnode{19}{\Huge{\cdots}} \\
     & & & & & \rnode{20}{\ddots}
\psset{nodesep=3pt,arcangle=15,labelsep=2ex,linewidth=0.3mm,arrows=->,arrowsize=1mm
3} \nccurve[angleA=160,angleB=200,ncurv=4]{0}{0}
\nccurve[angleA=70,angleB=110,ncurv=6]{1}{1}
\nccurve[angleA=70,angleB=110,ncurv=6]{3}{3}
\nccurve[angleA=70,angleB=110,ncurv=6]{6}{6}
\nccurve[angleA=70,angleB=110,ncurv=6]{10}{10}
\nccurve[angleA=200,angleB=240,ncurv=4]{2}{2}
\nccurve[angleA=200,angleB=240,ncurv=4]{5}{5}
\nccurve[angleA=200,angleB=240,ncurv=4]{9}{9}
\nccurve[angleA=200,angleB=240,ncurv=4]{14}{14}
\nccurve[angleA=25,angleB=65,ncurv=5]{4}{4}
\nccurve[angleA=25,angleB=65,ncurv=5]{8}{8}
\nccurve[angleA=25,angleB=65,ncurv=5]{13}{13}
\nccurve[angleA=25,angleB=65,ncurv=5]{7}{7}
\nccurve[angleA=25,angleB=65,ncurv=5]{11}{11}
\nccurve[angleA=25,angleB=65,ncurv=5]{12}{12}
\ncarc{0}{1}\ncarc{0}{2} \ncarc{1}{2}\ncarc{1}{4}
\ncarc{2}{5}\ncarc{5}{9}\ncarc{9}{14}\ncarc{14}{20}
\ncarc{3}{4}\ncarc{4}{5} \ncarc{6}{7}
\ncarc{1}{3}\ncarc{3}{6}\ncarc{6}{10}\ncarc{10}{15}\ncarc{2}{4}\ncarc{4}{7}\ncarc{7}{11}\ncarc{11}{16}\ncarc{5}{8}\ncarc{8}{12}\ncarc{12}{17}\ncarc{9}{13}\ncarc{13}{18}\ncarc{14}{19}
\ncarc{7}{8}\ncarc{8}{9}\ncarc{10}{11}\ncarc{11}{12}\ncarc{12}{13}\ncarc{13}{14}
\ncarc{3}{7}\ncarc{6}{11}\ncarc{4}{8}\ncarc{7}{12}\ncarc{8}{13}\ncarc{10}{16}\ncarc{11}{17}\ncarc{12}{18}\ncarc{13}{19}
%\psset{labelsep=-4.2ex}\nput{90}{0}{(0,0)}
\end{psmatrix}
$$
\end{center}
\caption{Diagram of all possible transitions of the discrete-time pure-birth QBD process generated by the Jacobi matrix $J_U$.}
\end{figure}
Again, the initial state is $(n,k)$, where $n$ is the number of blue balls inside the $k$-th urn $\mbox{A}_k$. Remove all the balls and put $k+\gamma+1$ blue balls and $k+\beta+1$ red balls in the urn $\mbox{A}_k$. Draw one ball from the urn at random with the uniform distribution. We have two possibilities:
\begin{enumerate}
\item If we get a blue ball then we remove/add balls until we have $n+k+\alpha+\beta+\gamma+3$ blue balls and $n-k$ red balls in the urn $\mbox{A}_k$. Then we draw again one ball from the urn at random with the uniform distribution and we have two possibilities:

\begin{itemize}
\item If we get a blue ball then we remove all balls in urn $\mbox{A}_{k}$ and add $n+1$ blue balls to the urn $\mbox{A}_{k+1}$ and start over. Therefore, joining both steps, we have
$$
\mathbb{P}\left[Z_1^{(2)}=(n+1,k+1)\; |\; Z_0^{(2)}=(n,k)\right]=\frac{\gamma+k+1}{\beta+\gamma+2k+2}\frac{n+k+\alpha+\beta+\gamma+3}{2n+\alpha+\beta+\gamma+3}.
$$
Observe that this probability is given by $x_{n,k}^{(3)}$ in \eqref{xyprob}.
\item If we get a red ball then we remove all balls in urn $\mbox{A}_{k}$ and add $n$ blue balls to the urn $\mbox{A}_{k+1}$ and start over. Therefore, joining both steps, we have
$$
\mathbb{P}\left[Z_1^{(2)}=(n,k+1)\; |\; Z_0^{(2)}=(n,k)\right]=\frac{\gamma+k+1}{\beta+\gamma+2k+2}\frac{n-k}{2n+\alpha+\beta+\gamma+3}.
$$
Observe that this probability is given by $y_{n,k}^{(3)}$ in \eqref{xyprob}.
\end{itemize}
\item If we get a red ball then we remove/add balls until we have $n-k+\alpha+1$ blue balls and $n+k+\beta+\gamma+2$ red balls in the urn $\mbox{A}_k$. Then we draw again one ball from the urn at random with the uniform distribution and we have two possibilities:
\begin{itemize}
\item If we get a blue ball then we remove/add balls in urn $\mbox{A}_{k}$ until we have $n+1$ blue balls in urn $\mbox{A}_{k}$ and start over. Therefore, joining both steps, we have
$$
\mathbb{P}\left[Z_1^{(2)}=(n+1,k)\; |\; Z_0^{(2)}=(n,k)\right]=\frac{\beta+k+1}{\beta+\gamma+2k+2}\frac{n-k+\alpha+1}{2n+\alpha+\beta+\gamma+3}.
$$
Observe that this probability is given by $x_{n,k}^{(2)}$ in \eqref{xyprob}.
\item If we get a red ball then we remove/add balls in urn $\mbox{A}_{k}$ until we have $n$ blue balls in urn $\mbox{A}_{k}$ and start over. Therefore, joining both steps, we have
$$
\mathbb{P}\left[Z_1^{(2)}=(n,k)\; |\; Z_0^{(2)}=(n,k)\right]=\frac{\beta+k+1}{\beta+\gamma+2k+2}\frac{n+k+\beta+\gamma+2}{2n+\alpha+\beta+\gamma+3}.
$$
Observe that this probability is given by $y_{n,k}^{(2)}$ in \eqref{xyprob}.
\end{itemize}
\end{enumerate}

The urn model for $J_{2}$ will be the composition of Experiment 1 and then Experiment 2. Combining all possibilities we have the transition probabilities for the QBD process $\{Z_t : t=0,1,\ldots\}$. Indeed
\begin{align*}
\mathbb{P}\left[Z_1=(n+1,k-1)\; |\; Z_0=(n,k)\right]&=s_{n,k}^{(1)}x_{n,k-1}^{(2)}=a_{n,k}^{(1)},\\
\mathbb{P}\left[Z_1=(n+1,k)\; |\; Z_0=(n,k)\right]&=s_{n,k}^{(1)}x_{n,k-1}^{(3)}+s_{n,k}^{(2)}x_{n,k}^{(2)}=a_{n,k}^{(2)},\\
\mathbb{P}\left[Z_1=(n+1,k+1)\; |\; Z_0=(n,k)\right]&=s_{n,k}^{(2)}x_{n,k}^{(3)}=a_{n,k}^{(3)},\\
\mathbb{P}\left[Z_1=(n,k-1)\; |\; Z_0=(n,k)\right]&=s_{n,k}^{(1)}y_{n,k-1}^{(2)}+r_{n,k}^{(1)}x_{n-1,k-1}^{(2)}=b_{n,k}^{(1)},\\
\mathbb{P}\left[Z_1=(n,k)\; |\; Z_0=(n,k)\right]&=s_{n,k}^{(1)}y_{n,k-1}^{(3)}+s_{n,k}^{(2)}y_{n,k}^{(2)}+r_{n,k}^{(1)}x_{n-1,k-1}^{(3)}+r_{n,k}^{(2)}x_{n-1,k}^{(2)}=b_{n,k}^{(2)},\\
\mathbb{P}\left[Z_1=(n,k+1)\; |\; Z_0=(n,k)\right]&=s_{n,k}^{(2)}y_{n,k}^{(3)}+r_{n,k}^{(2)}x_{n-1,k}^{(3)}=b_{n,k}^{(3)},\\
\mathbb{P}\left[Z_1=(n-1,k-1)\; |\; Z_0=(n,k)\right]&=r_{n,k}^{(1)}y_{n-1,k-1}^{(2)}=c_{n,k}^{(1)},\\
\mathbb{P}\left[Z_1=(n-1,k)\; |\; Z_0=(n,k)\right]&=r_{n,k}^{(1)}y_{n-1,k-1}^{(3)}+r_{n,k}^{(2)}y_{n-1,k}^{(2)}=c_{n,k}^{(2)},\\
\mathbb{P}\left[Z_1=(n-1,k+1)\; |\; Z_0=(n,k)\right]&=r_{n,k}^{(2)}y_{n-1,k}^{(3)}=c_{n,k}^{(3)}.
\end{align*}
As we showed before, and since we are assuming that $\alpha,\beta,\gamma$ are nonnegative integers, the urn model derived by this discrete-time QBD process is always transient.

A similar continuous-time QBD process could have been derived for the normalization of the polynomials at the point $(0,0)$ in Section \ref{secnorm0}, but now with two parameters $\tau_1,\tau_2$ subject to the restrictions in \eqref{condaa22}.

\section{Concluding remarks and further research}\label{secult}

In this paper we have studied several examples of bivariate orthogonal polynomials related to discrete or continuous-time QBD processes. Also, we gave probabilistic models for them in terms of, mainly, urn models. All examples are constructed according to certain normalization of the polynomials at one of the ``corners'' of the support of orthogonality. This restriction seems to be important in order to have recurrence relations with probabilistic interpretations (like the situation of scalar birth-death chains), but not all points in the boundary (including corners) lead to coefficients which may be interpreted as a QBD process, as we saw, for instance, in the case of orthogonal polynomials on the triangle. One open problem could be trying to explain why this restriction is needed in order to construct a QBD process.

Certainly, we have analyzed other examples of bivariate orthogonal polynomials. In particular, the seven different classes studied by T. Koornwinder in \cite{Ko75}. But we have not found any probabilistic interpretation in terms of QBD processes in any of them. The two main reasons for that are:
\begin{enumerate}
\item The bivariate orthogonal polynomials $\{P_{n,k} : 0\leq k\leq n \}$ (any way of constructing them) can not be normalized at some interesting point $(a,b)$ at the boundary of the support of orthogonality such that $P_{n,k}(a,b)=1$ for all $n\in\mathbb{N}_0$ and $0\leq k\leq n$, since they may vanish at that point for some degree of the polynomials. In this situation, we can not proceed in the same way as we have proceeded through this paper. It is possible, though, that there may exist another normalization of the polynomials such that the coefficients of the three-term recurrence relations can be linearly combined in such a way that they lead to a probabilistic interpretation (for instance, for the product Laguerre polynomials in Section \ref{secplpl}), but we have not found any nontrivial situation where this happens.
\item It is possible to normalize the bivariate orthogonal polynomials $\{P_{n,k} : 0\leq k\leq n \}$ at certain point $(a,b)$ at the boundary (or inside) the support of orthogonality such that $P_{n,k}(a,b)=1$ for all $n\in\mathbb{N}_0$ and $0\leq k\leq n$, but there are no possible linear combinations of the two corresponding Jacobi matrices such that they lead to a QBD process. This is the situation, for instance, for the product Jacobi-Laguerre polynomials at the point $(0,1)$ (see Remark \ref{remJL}) or any other example normalized at some point which is not a ``corner'' of the support of orthogonality.
\end{enumerate}

There is one iconic example that we have not been able to find any probabilistic interpretation for, namely orthogonal polynomials on the unit disk. We have tried several definitions and normalizations of the polynomials, but it seems that neither of them works out due to some of the two reasons mentioned above. We believe that the problem with this example may lie in the fact that the unit disk does not have any ``corners''.

There are many examples of bivariate orthogonal polynomials that have not been considered in this paper, like, for instance, the two families of Koornwinder polynomials (see Sections 2.7 and 2.9 of \cite{DX14}) or some families of Krall or Sobolev type bivariate polynomials (see for example \cite{AX13,DFI20,DFLPP,DFPP12,DFPP16,DFP18,FMPP15,MP,Xu17}), where a Dirac delta is added at one (or several) points of the support of orthogonality or some other more complicated situations. Also we have not considered examples of multivariate orthogonal polynomials for $d\geq3$. For instance, three-dimensional examples, like the unit ball, the unit sphere or the simplex. In this case we have $d=3$ in \eqref{knd} and we will have diagrams similar to the one in Figure 3, but now the number of phases is $\binom{n+2}{2}$. Certainly some of the previous problems will be dealt with in future publications.

\end{document}